\documentclass{amsart}
\usepackage{multirow}
\usepackage{mathrsfs}
\usepackage{amsfonts}
\usepackage{caption}
\usepackage{cases}
\usepackage{url}
\usepackage{color}
\newcommand{\re}{\mathbb{R}}
\newcommand{\cpx}{\mathbb{C}}

\newcommand{\N}{\mathbb{N}}

\renewcommand{\P}{\mathbb{P}}

\newcommand{\lmd}{\lambda}

\newcommand{\dt}{\delta}

\newcommand{\mF}{\mathcal{F}}
\newcommand{\mA}{\mathcal{A}}
\newcommand{\mB}{\mathcal{B}}
\def\af{\alpha}

\def\rank{\mbox{rank}}
\newcommand{\sig}{\sigma}
\newcommand{\Sig}{\Sigma}

\newcommand{\reff}[1]{(\ref{#1})}

\newcommand{\prm}{\prime}

\newcommand{\mc}[1]{\mathcal{#1}}
\newcommand{\mt}[1]{\mathtt{#1}}

\newcommand{\supp}[1]{\mbox{supp}(#1)}

\newcommand{\bdes}{\begin{description}}
\newcommand{\edes}{\end{description}}

\newcommand{\bnum}{\begin{enumerate}}
\newcommand{\enum}{\end{enumerate}}

\newcommand{\bit}{\begin{itemize}}
\newcommand{\eit}{\end{itemize}}

\newcommand{\bea}{\begin{eqnarray}}
\newcommand{\eea}{\end{eqnarray}}
\newcommand{\be}{\begin{equation}}
\newcommand{\ee}{\end{equation}}

\newcommand{\baray}{\begin{array}}
\newcommand{\earay}{\end{array}}

\newcommand{\bca}{\begin{cases}}
\newcommand{\eca}{\end{cases}}

\newcommand{\bcen}{\begin{center}}
\newcommand{\ecen}{\end{center}}

\newcommand{\bbm}{\begin{bmatrix}}
\newcommand{\ebm}{\end{bmatrix}}

\newcommand{\bmx}{\begin{matrix}}
\newcommand{\emx}{\end{matrix}}

\newcommand{\bpm}{\begin{pmatrix}}
\newcommand{\epm}{\end{pmatrix}}

\newcommand{\btab}{\begin{tabular}}
\newcommand{\etab}{\end{tabular}}

\newcommand{\balg}{\begin{algorithm}}
\newcommand{\ealg}{\end{algorithm}}
\newcommand{\br}{\begin{remark}}
\newcommand{\er}{\end{remark}}
\newcommand{\bex}{\begin{example}}
\newcommand{\eex}{\end{example}}

\newtheorem{theorem}{Theorem}[section]
\newtheorem{lemma}[theorem]{Lemma}
\newtheorem{proposition}[theorem]{Proposition}
\newtheorem{prop}[theorem]{Proposition}

\theoremstyle{definition}
\newtheorem{definition}[theorem]{Definition}
\newtheorem{example}[theorem]{Example}

\newtheorem{algorithm}[theorem]{Algorithm}
\newtheorem{remark}[theorem]{Remark}

\numberwithin{equation}{section}

\begin{document}

\title[Tensor Eigenvalue Complementarity Problems]
{Tensor Eigenvalue Complementarity Problems}

%    Remove any unused author tags.

%    author one information

\author{Jinyan Fan}
\address{School of Mathematical Sciences, and MOE-LSC, Shanghai Jiao Tong University,
Shanghai 200240, P.R. China}
\email{jyfan@sjtu.edu.cn}

\author{Jiawang Nie}
\address{
Department of Mathematics,  University of California San Diego,  9500
Gilman Drive,  La Jolla,  California 92093,  USA.}
\email{njw@math.ucsd.edu}

\author{Anwa Zhou}
\address{
Department of Mathematics, Shanghai University,
Shanghai 200444, P.R. China}
\email{zhouanwa@shu.edu.cn}

\thanks{Jinyan Fan was partially supported by the NSFC grants 11171217 and 11571234.
Jiawang Nie was partially supported by the NSF grants DMS-1417985 and DMS-1619973.
Anwa Zhou was partially supported by the CPSF grants BX201600097 and 2016M601562.
}

%    For articles to be published after 1 January 2010, you may use
%    the following version:
%\subjclass[2010]{Primary }


\begin{abstract}
This paper studies tensor eigenvalue complementarity problems.
Basic properties of standard and complementarity
tensor eigenvalues are discussed.
We formulate tensor eigenvalue complementarity problems
as constrained polynomial optimization.
When one tensor is strictly copositive,
the complementarity eigenvalues can be computed by solving
polynomial optimization with normalization by strict copositivity.
When no tensor is strictly copositive,
we formulate the tensor eigenvalue complementarity problem
equivalently as polynomial optimization by a randomization process.
The complementarity eigenvalues can be computed sequentially.
The formulated polynomial optimization can be solved by
Lasserre's hierarchy of semidefinite relaxations.
We show that it has finite convergence for generic tensors.
Numerical experiments are presented to show
the efficiency of proposed methods.
\end{abstract}

\maketitle

\section{Introduction}

Let $\re$ be the real field, $\mathbb{R}^n$
be the space of all real $n$-dimensional vectors,
and $\mathbb{R}^{n\times n}$ be the space of all real $n$-by-$n$ matrices.
Denote by $\re_+^n$ the nonnegative orthant, i.e., the set of vectors
in $\mathbb{R}^n$ whose entries are all nonnegative.

The classical matrix eigenvalue complementarity problem (MEiCP) is that:
for given two matrices $A, B\in \mathbb{R}^{n\times n}$, we want to
find a number $\lambda\in \mathbb{R}$ and a nonzero vector
$x \in \mathbb{R}^n$ such that
\be
%\tag{EiCP}
\label{EiCP}
0 \leq x \perp (\lambda Bx - Ax) \geq 0.
\ee
In the above, $a \perp b$ means that the two vectors $a,b$
are perpendicular to each other.
For $(\lambda, x)$ satisfying \reff{EiCP}, $\lambda$ is called a
complementary eigenvalue of $(A, B)$ and
$x$ is called the associated complementary eigenvector.
MEiCPs have wide applications,
such as static equilibrium states of mechanical systems
with unilateral friction \cite{Pinto01},
the dynamic analysis of structural mechanical systems
\cite{Martins99,Martins00} and the contact problem in mechanics \cite{Martins04}.
The MEiCP \reff{EiCP} has at least one solution if
$ x^T B x \neq 0$ for all $x \in \re_+^n \setminus \{0\}$
(cf.~\cite{Judice09,Pinto10}).
When $A$ and $B$ are symmetric, the problem \reff{EiCP}
can be reduced to finding a stationary point of
the quotient $x^TAx/x^TBx$ over the standard simplex. For such cases,
nonlinear optimization methods can
be applied to solve MEiCPs
(cf.~\cite{Judice08,Queiroz04}).
When $A,B$ are not symmetric,
other approaches were proposed for solving MEiCPs,
such as the branch-and-bound technique \cite{Judice07,Judice09},
the scaling-and-projection and the power iteration
\cite{Pinto09, Pinto10},
semismooth Newton-type methods \cite{Adly11, Adly13}.
Most existing methods aim at
computing one of the complementarity eigenvalues.
The matrix complementarity problem is NP-hard \cite{Judice07}.

Eigenvalues were recently studied for tensors \cite{DingWei15,HHLQ13,Lim2005,Qi2005}.
For an integer $m>0$, an {\it $m$-th order $n$-dimensional} tensor $\mathcal{A}$ is
a multi-array indexed as
\[
\mathcal{A} \, := \, (\mathcal{A}_{i_1,\ldots,i_m})_{1 \leq i_1, \ldots , i_m \leq n}.
\]
Let $\mathrm{T}^m(\mathbb{R}^{n})$ be the space of all such real tensors.
For $x := (x_1,\ldots,x_n) \in \re^n$, denote by $\mathcal{A}x^{m-1}$
the vector in $\re^n$ such that, for each $i=1,2,\ldots,n$,
\be \label{vc:Ax:m-1}
(\mathcal{A}x^{m-1})_i=\sum_{i_2,\ldots , i_m=1}^{n}
\mathcal{A}_{i,i_2,\ldots,i_m} x_{i_2}\cdots x_{i_m}.
\ee
Denote by $\mathcal{A}x^{m}$ the homogeneous polynomial
\[
\mathcal{A}x^{m}=\sum_{i_1,i_2,\ldots , i_m=1}^{n}
\mathcal{A}_{i_1,i_2,\ldots,i_m} x_{i_1}x_{i_2}\cdots x_{i_m}.
\]
Clearly,
$
\mathcal{A}x^{m}= \sum_{j =1}^{n} x_j (\mathcal{A}x^{m-1})_j.
$
Lim \cite{Lim2005} and Qi \cite{Qi2005} introduced
the notion of tensor eigenvalues.
Generalized eigenvalues can be defined similarly for tensors \cite{DingWei15}.
For two nonzero tensors $\mathcal{A},\mathcal{B}\in \mathrm{T}^m(\mathbb{R}^{n})$,
if a pair $(\lambda,x)\in \mathbb{C}\times (\mathbb{C}^n \setminus \{0\})$
satisfies the equation
\begin{equation}\label{eq:eigendef}
 \lambda \mathcal{B}x^{m-1} -\mathcal{A}x^{m-1}=0,
\end{equation}
then $\lambda$ is called a $\mathcal{B}$-eigenvalue of $\mathcal{A}$
and $x$ is the associated $\mathcal{B}$-eigenvector.
Such $(\lmd, x)$ is called a $\mathcal{B}$-eigenpair.
Recently, Cui, Dai and Nie~\cite{Cui2014} studied
$\mathcal{B}$-eigenvalues of symmetric tensors.
They proposed a semidefinite relaxation approach
for computing all real $\mathcal{B}$-eigenvalues sequentially,
from the largest to the smallest.
Each eigenvalue can be computed by solving a finite hierarchy of semidefinite relaxations.
This approach was originally used for computing the hierarchy
of local minimums for polynomial optimization \cite{Nie-local}.

Recently, Ling et al.~\cite{Ling2015} introduced
the tensor eigenvalue complementarity problem (TEiCP):
for two given tensors $\mathcal{A},\mathcal{B}\in \mathrm{T}^m(\mathbb{R}^{n})$,
we want to find a number $\lambda\in \mathbb{R}$
and a nonzero vector $x \in\mathbb{R}^n$ such that
\begin{equation}
\label{TEiCP}
0\leq x \perp ( \lambda \mathcal{B} x^{m-1} - \mathcal{A} x^{m-1} ) \geq 0.
\end{equation}
For such a pair $(\lambda, x)$, $\lambda$ is called a
complementary eigenvalue of $(\mathcal{A},\mathcal{B})$
and $x$ is called the associated complementary eigenvector.
For convenience, the complementary eigenvalues and eigenvectors
are respectively called {\it C-eigenvalues and C-eigenvectors}.
The above $(\lmd, x)$ is called a {\it C-eigenpair}.
Clearly, when $m=2$, the TEiCP is reduced to the classical
matrix eigenvalue complementarity problem.
TEiCPs have wide applications such as
higher-order Markov chains \cite{Ng2009},
magnetic resonance imaging \cite{Qi2010}.
We refer to \cite{Chen2015,Ling2015}
for more applications of TEiCPs.

In the existing references (cf.~\cite{Ling2015}),
C-eigenvalues defined as in \reff{TEiCP}
are also called {\it Pareto-eigenvalues}. Indeed, Ling et al.~\cite{Ling2015}
considered more general tensor eigenvalue complementarity problems,
where the conditions $x \geq 0$ and $\lmd \mB x^{m-1} - \mA x^{m-1} \geq 0$
are replaced by
\[
x \in K, \quad
\lmd \mB x^{m-1} - \mA x^{m-1} \in K^*.
\]
Here, $K$ is a closed convex cone and $K^*$ is the dual cone.
In \cite{Ling2015}, it was shown that the TEiCP has at least one solution,
under the assumption that
$\mathcal{B}x^m \neq 0$ for all $ x \in \mathbb{R}_+^n \setminus \{0\}$.
They also gave an upper bound for the number of C-eigenvalues,
for nonsingular tensor pairs $(\mA, \mB)$ (see \S\ref{ssc:teigvec}
for the definition). Moreover,
a scaling-and-projection algorithm was given for solving TEiCPs.
Recently, Chen et al.~\cite{Chen2015} have further new work on TEiCPs.
When the tensors are symmetric, they
reformulated the problem as nonlinear optimization and
then proposed a shifted projected power method.
Chen and Qi~\cite{ChenQi2015} reformulated the TEiCP as a system of
nonlinear equations and proposed a damped semi-smooth Newton method for solving it.
Some properties of Pareto-eigenvalues are further studied in \cite{Xu2015}.
Generally, the tensor eigenvalue complementarity problem is difficult to solve.
It is also NP-hard, since the TEiCP includes the MEiCP as a special case.

\bigskip
\noindent
{\bf Contributions}\,
In this paper, we study how to solve TEiCPs.
Our aim is to compute all C-eigenvalues,
if there are finitely many ones.
We formulate TEiCPs equivalently as polynomial optimization problems,
and then solve them by Lasserre type semidefinite relaxations.
Throughout the paper, a property is said to
be {\it generically} true
in a tensor space if it holds in an open dense subset of that space,
in the Zariski topology. For such a property,
a tensor in that open dense set is called a {\it generic} tensor.

First, we study properties of generalized eigenvalues
of tensor pairs. For nonsingular tensor pairs,
it is known that the number of eigenvalues is finite
(cf.~\cite{DingWei15}). For generic tensors,
we show a further new result: for each eigenvalue, there is
a unique eigenvector, up to scaling. Thus, the number
of normalized eigenvectors is also finite.
Similarly, for generic tensors,
we can also show that the number of C-eigenvalues and
C-eigenvectors (up to scaling) are finite.
These results are given in Section~\ref{sc:pro:teig}.

Second, we show how to solve tensor eigenvalue complementarity problems
when the tensor $\mB$ is strictly copositive
(i.e., $\mB x^m >0$ for all $ x \in \mathbb{R}_+^n \setminus \{0\}$).
For such cases, the complementarity eigenvectors
can be normalized such that $\mB x^m = 1$.
Then, we formulate the problem as constrained polynomial optimization.
The complementarity eigenvalues can be computed sequentially,
from the smallest to the biggest. Each of them
can be solved by a sequence of semidefinite relaxations.
We prove that such sequence has finite convergence for generic tensors,
subject to that $\mB$ is strictly copositive.
This will be shown in Section~\ref{sect:cop}.

Third, we study how to solve tensor eigenvalue complementarity problems
when $\mB$ is not not copositive. For such tensors, a C-eigenvector $x$
may not be normalized as $\mB x^m = 1$.
Thus, we formulate TEiCPs as polynomial optimization
in a different way. By a randomization process,
the complementarity eigenvectors are classified in two cases.
For each case, the TEiCP is equivalently formulated as
a polynomial optimization problem.
The C-eigenvectors can be computed in order,
by choosing a randomly chosen objective.
Each of them can be computed by a sequence of semidefinite relaxations.
For generic tensors, we show that it converges in finitely many steps.
The results are shown in Section~\ref{sc:gTEiCP}.

In Section~\ref{sc:numexp}, we present numerical experiments for
solving tensor eigenvalue complementarity problems.
Some preliminaries in polynomial optimization and
moment problems are given in Section~\ref{sc:prelim}.

\section{Preliminaries}
\label{sc:prelim}

\noindent
{\bf Notation}\,
The symbol $\N$ (resp., $\re$, $\cpx$) denotes the set of
nonnegative integers (resp., real, complex numbers).
For integer $n>0$, $[n]$ denotes the set $\{1,\ldots,n\}$.
For two vectors $a,b \in \re^n$, $a \circ b$ denotes the
Hadamard product of $a$ and $b$, i.e.,
the product is defined componentwise.
For $x=(x_1,\ldots,x_n)$ and $\af = (\af_1, \ldots, \af_n)$,
denote the monomial power
\[
x^\af := x_1^{\af_1}\cdots x_n^{\af_n}.
\]
The symbol $[x]_d$ denotes the following vector of monomials
\[
[x]_d^T = [\, 1 \quad  x_1 \quad \cdots \quad x_n \quad x_1^2 \quad
x_1x_2 \quad \cdots \cdots
\quad x_1^d \quad x_1^{d-1}x_2 \quad \cdots \cdots \quad x_n^d \,],
\]
The symbol $\re[x] := \re[x_1,\ldots,x_n]$
denotes the ring of polynomials in $x:=(x_1,\ldots,x_n)$
and with real coefficients.
The ring $\cpx[x] := \cpx[x_1,\ldots,x_n]$
is similarly defined over the complex field.
The $deg(p)$ denotes the degree of a polynomial $p$.
The cardinality of a set $S$ is denoted as $|S|$.
For $t\in \re$, $\lceil t\rceil$ (resp., $\lfloor t\rfloor$)
denotes the smallest integer not smaller
(resp., the largest integer not bigger) than $t$.
For a matrix $A$, $A^T$ denotes its transpose.
For a symmetric matrix $X$, $X\succeq 0$ (resp., $X\succ 0$) means
$X$ is positive semidefinite (resp., positive definite).
For a vector $u$, $\| u \|$ denotes its standard Euclidean norm.
The $e_i$ denotes the standard $i$-th unit vector in $\N^n$.

\subsection{Polynomial optimization}

In this section, we review some basics in polynomial optimization.
We refer to \cite{Lasserre01,Lasserre09, Laurent} for more details.

An ideal $I$ in $\re[x]$ is a subset of $\re[x]$
such that $ I \cdot \re[x] \subseteq I$
and $I+I \subseteq I$. For a tuple $h=(h_1,\ldots,h_m)$ in $\re[x]$,
denote the ideal
\[
I(h) :=
h_1 \cdot \re[x] + \cdots + h_m  \cdot \re[x].
\]
The $k$-th {\it truncation} of the ideal $I(h)$,
denoted as $I_k(h)$, is the set
\be \label{Ik}
h_1 \cdot \re[x]_{k-\deg(h_1)} + \cdots + h_m  \cdot \re[x]_{k-\deg(h_m)}.
\ee
In the above, $\mathbb{R}[x]_t$ is the set of polynomials
in $\re[x]$ with degrees at most $t$.
Clearly, $I(h)=\cup_{k\in \mathbb{N}} I_{k}(h)$.

A polynomial $\psi$ is said to be a sum of squares (SOS)
if $\psi = q_1^2+\cdots+ q_k^2$ for some $q_1,\ldots, q_k \in \re[x]$.
The set of all SOS polynomials in $x$ is denoted as $\Sig[x]$.
For a degree $m$, denote the truncation
\[
\Sig[x]_m := \Sig[x] \cap \re[x]_m.
\]
For a tuple $g=(g_1,\ldots,g_t)$,
its {\it quadratic module} is the set
\[
Q(g):=  \Sig[x] + g_1 \cdot \Sig[x] + \cdots + g_t \cdot \Sig[x].
\]
The $k$-th truncation of $Q(g)$ is the set
\be \label{Qk}
Q_k(g) :=
\Sig[x]_{2k} + g_1 \cdot \Sig[x]_{d_1} + \cdots + g_t \cdot \Sig[x]_{d_t}
\ee
where each $d_i = 2k - \deg(g_i)$.
Note that $Q(g)= \cup_{k\in \mathbb{N}} Q_k (g)$.

The set $I(h)+ Q(g)$ is said to be archimedean if there exists
$N>0$ such that $N-\|x\|^2\in I(h)+Q(g)$.
For the tuples $h,g$ as above, denote
\be \label{E(h):S(g)}
E(h) := \{x\in \mathbb{R}^{n} \mid \ h(x)= 0 \}, \qquad
S(g) := \{x\in \mathbb{R}^{n} \mid \ g(x) \geq 0\}.
\ee
Clearly, if $I(h)+ Q(g)$ is archimedean, then the set $E(h) \cap S(g)$ is compact.
On the other hand, if $E(h) \cap S(g)$ is compact,
then $I(h)+ Q(g)$ can be forced to be archimedean
by adding the polynomial $M-\|x\|^2$ to the tuple $g$,
for sufficiently large $M$.

If $f \in I(h)+Q(g)$, then $f\geq 0$ on the set $E(h) \cap S(g)$.
Conversely, if $f > 0$ on $E(h) \cap S(g)$
and $I(h)+ Q(g)$ is archimedean, then $f \in I(h)+Q(g)$.
This is called {\it Putinar's Positivstellensatz}
(cf.~\cite{Putinar}) in the literature.
Interestingly, when $f$ is only nonnegative on $E(h) \cap S(g)$,
we also have $f \in I(h)+Q(g)$,
if some standard optimality conditions hold (cf.~\cite{Nie3}).

\subsection{Moment and localizing matrices}
\label{ssc:momloc}

For $\alpha=(\alpha_1,\ldots,\alpha_n)$,
denote $|\alpha| :=\alpha_1+\ldots+\alpha_n$ and
$$
\mathbb{N}_d^{n} := \{ \alpha \in \mathbb{N}^{n}: \, |\alpha| \leq d\}.
$$
Let $\re^{\N_d^n}$ be the space of real vectors indexed by $\af \in \N^n_d$.
A vector in $\re^{\N_d^n}$ is called a
{\it truncated multi-sequence} (tms) of degree $d$.
For $y \in \re^{\N_d^n}$, define the operation
\be \label{df:<p,y>}
\big \langle \sum_{\af \in \N_d^n}
p_\af x_1^{\af_1}\cdots x_n^{\af_n},
y  \big \rangle := \sum_{\af \in \N_d^n}  p_\af y_\af.
\ee
(In the above, each $p_\af$ is a coefficient.)
We say that $y$ admits a representing measure supported in a set $T$
if there exists a Borel measure $\mu$ such that
its support, denoted as $\supp{\mu}$, is contained in $T$ and
\[
y_\af = \int_T  x^\af \mt{d} \mu \quad \forall \af \in \N_d^n.
\]

For a polynomial $q \in \re[x]_{2k}$,
the $k$-th {\it localizing matrix} of $q$,
generated by a tms $y \in \re^{\N^n_{2k}}$,
is the symmetric matrix $L_q^{(k)}(y)$ satisfying
\[
vec(p_1)^T \Big( L_q^{(k)}(y) \Big) vec(p_2)  = \langle q p_1p_2, y \rangle,
\]
for all $p_1,p_2 \in \re[x]$ with
$\deg(p_1), \deg(p_2) \leq k - \lceil \deg(q)/2 \rceil$.
In the above, $vec(p_i)$ denotes the coefficient vector of the polynomial $p_i$.
When $q = 1$ (the constant one polynomial),
$L_q^{(k)}(y)$ becomes a {\it moment matrix} and is denoted as
\be \label{moment:mat}
M_k(y):= L_{1}^{(k)}(y).
\ee
When $q=(q_1,\ldots,q_r)$ is a tuple of $r$ polynomials, then we denote
\be \label{mat:locliz}
L_q^{(k)} (y) := \left(L_{q_1}^{(k)} (y),\ldots, L_{q_r}^{(k)} (y)\right).
\ee
We refer to \cite{CurtoF,Helton,Nie1} for localizing and moment matrices.

Let $h = (h_1, \ldots, h_m)$ and
$g=(g_1,\ldots, g_t)$ be two polynomial tuples.
In applications, people are often interested in
whether or not a tms
$y \in \re^{\N^n_{2k}}$ admits a representing measure
whose support is contained in
$E(h) \cap S(g)$, as in \reff{E(h):S(g)}. For this to be true,
a necessary condition (cf.~\cite{CurtoF,Helton}) is that
\be \label{Lg>=0}
M_k(y)\succeq 0, \quad L_{g}^{(k)}(y) \succeq 0, \quad
L_{h}^{(k)}(y) = 0.
\ee
However, the above is typically not sufficient. Let
\[
d_0 = \max \, \{1,  \lceil \deg(h)/2 \rceil, \lceil \deg(g)/2 \rceil \}.
\]
If $y$ satisfies \reff{Lg>=0} and the rank condition
\be \label{con:fec}
\rank \, M_{k-d_0}(y) \, = \, \rank \, M_k(y),
\ee
then $y$ admits a measure supported in $E(h) \cap S(g)$ (cf.~\cite{CurtoF}).
In such case, $y$ admits a unique finitely atomic
measure on $E(h) \cap S(g)$.
For convenience, we just call that $y$ is {\it flat} with respect to
$h=0$ and $g\geq 0$
if \reff{Lg>=0} and \reff{con:fec} are {\it both} satisfied.

For $t\leq d$ and $w\in \mathbb{R}^{\mathbb{N}^{n}_{d}}$,
denote the truncation of $w$:
\be \label{trun:z|t}
w|_t \, = \, (w_{\alpha})_{\alpha\in \mathbb{N}^{n}_{t}}.
\ee
For two tms' $y \in \re^{\N^n_{2k} }$ and $z \in \re^{\N^n_{2l} }$ with $k<l$,
we say that $y$ is a truncation of $z$
(equivalently, $z$ is an extension of $y$), if $y=z|_{2k}$.
For such case, $y$ is called a flat truncation of $z$ if $y$ is flat,
and $z$ is a flat extension of $y$ if $z$ is flat.
Flat extensions and flat truncations are very useful
in solving polynomial optimization
and truncated moment problems (cf.~\cite{Nie2,Nie,Nie1}).

\section{Properties of tensor eigenvalues}
\label{sc:pro:teig}

This section studies some properties of
standard eigenvalues and complementarity eigenvalues,
for generic tensor pairs.

\subsection{Tensor eigenvalues and eigenvectors}
\label{ssc:teigvec}

For two given tensors $\mathcal{A},\mathcal{B}\in \mathrm{T}^m(\cpx^{n})$,
a number $\lmd \in \cpx$ is called a generalized eigenvalue of the pair
$(\mA, \mB)$ if there exists a vector $x \in \cpx^n\setminus \{0\}$ such that
\begin{equation}\label{eq:eigendef}
  \mathcal{A}x^{m-1}-\lambda \mathcal{B}x^{m-1}=0.
\end{equation}
If so, such $x$ is called a generalized eigenvector, associated with $\lmd$,
of the pair $(\mA, \mB)$.
We refer to Ding and Wei \cite{DingWei15} for generalized tensor eigenvalues.
For convenience, we just call that the above $\lmd$ (resps., $x$)
is an eigenvalue (resp., eigenvector) of $(\mA,\mB)$,
and $(\lmd, x)$ is called an eigenpair.

Tensor eigenvalues are closely related to the notion of {\it resultant},
denoted as $Res$, for tuples of homogeneous polynomials.
For a tuple $f=(f_1, \ldots , f_n$) of $n$ homogeneous polynomials in
$x :=(x_1,\ldots,x_n)$, its resultant is the polynomial
$Res(f)$, in the coefficients of $f$, such that
$Res(f) = 0$ if and only if the homogeneous equation
\[ f_1(x)=\cdots=f_n(x)=0\]
has a nonzero solution in $\cpx^n$.
The $Res(f)$ is an irreducible polynomial,
and is homogeneous in the coefficients of each $f_i$.
We refer to Cox, Little and O'Shea \cite{Cox2006} for resultants.
For a tensor $\mF \in \mathrm{T}^m(\cpx^{n})$,
$\mF x^{m-1}$ is a tuple of
$n$ homogeneous polynomials of degree $m-1$.
For convenience, denote the resultant:
\be \label{df:R(F)}
R(\mF) \, := \,  Res( \mF x^{m-1} ).
\ee
Clearly, $\lmd$ is an eigenvalue of $(\mA, \mB)$ if and only if
\[
R( \mA - \lmd \mB) = 0.
\]
Note that $R( \mA - \lmd \mB)$ is a polynomial
in $\lmd$ and its degree is $n(m-1)^{n-1}$.
As in \cite{DingWei15}, $(\mA, \mB)$ is called a nonsingular tensor pair if the equation
\[
\mA x^{m-1} = \mB x^{m-1} = 0
\]
has the only zero solution.
Clearly, if $R(\mB) \neq 0$ then $(\mA, \mB)$ is nonsingular.

\begin{theorem}\label{thm:nb:eigvec}
Let $\mathcal{A},\mathcal{B}\in \mathrm{T}^m(\cpx^n)$ and $D := n (m-1)^{n-1}$.
\bit

\item [(i)] (\cite[Theorem~2.1]{DingWei15})
If $R(\mB) \neq 0$, then
$(\mA, \mB)$ has $D$ eigenvalues, counting multiplicities.

\item [(ii)] If $\mA,\mB$ are generic tensors in $\mathrm{T}^m(\cpx^n)$,
then $(\mA, \mB)$ has $D$ distinct eigenvalues. Moreover,
for each eigenvalue, there is a unique eigenvector, up to scaling.

\eit

\end{theorem}
\begin{proof}
(i) This item can be found in Theorem~2.1 of Ding and Wei \cite{DingWei15}.
If $R(\mB) \neq 0$, then $(\mA, \mB)$ is a nonsingular tensor pair.

(ii) The resultant $R(\mF)$ is an irreducible polynomial in
the entries of $\mF$. The hypersurface
\[
\mathscr{H} = \{\mF  \in \mathrm{T}^m(\cpx^n): \, R(\mF) = 0 \}
\]
is irreducible in the space $\mathrm{T}^m(\cpx^n)$.
Its minimum degree defining polynomial is $R(\mF)$, with the degree $D$.
The hypersurface $\mathscr{H}$ is smooth,
except a subset $\mathscr{E} \subseteq \mathscr{H}$
whose dimension is smaller than that of $\mathscr{H}$.
For generic $\mA, \mB$, the line
\[
\mathscr{L} = \{ \mA - \lmd \mB:  \, \lmd \in \cpx \}
\]
does not intersect the set $\mathscr{E}$. That is,
$\mathscr{L}$ intersects $\mathscr{H}$ only at smooth points of $\mathscr{H}$
(i.e., the intersection is transversal).
This implies that for all $\lmd$ satisfying
\[
\phi(\lmd) := R( \mA-\lmd \mB ) = 0,
\]
we have $\phi^\prm(\lmd) \ne 0$. The roots of $\phi$ are all simple.
Therefore, $(\mA, \mB)$ has $D$ distinct eigenvalues,
when $\mA,\mB$ are generic tensors in $\mathrm{T}^m(\cpx^n)$.

Let $X$ be the determinantal projective variety
\[
X = \big\{ x\in   \P^{n-1} : \,
\rank \bbm \mA x^{m-1} & \mB x^{m-1} \ebm  < 2 \big\}.
\]
(The $\P^{n-1}$ is the projective space of equivalent classes
of vectors in $\cpx^n$.  )
Clearly, if $R(\mathcal{B}) \neq 0$,
then $(\lmd, x)$ is an eigenpair of $(\mA,\mB)$ if and only if $x \in X$.
When $\mA,\mB$ are generic, we have $R(\mB) \ne 0$,
and the set $X$ is zero-dimensional (i.e., $X$ is a finite set),
and its cardinality is equal to the number $D$.
This can be implied by Propositions~A.5, A.6 of \cite{NieRan09}.

When $\mA,\mB$ are generic tensors,
$(\mA,\mB)$ has $D$ distinct eigenvalues.
For each eigenvalue, there is at least one eigenvector.
This implies that there is a unique eigenvector up to scaling.
\end{proof}

\subsection{Combinatorial eigenvalues and eigenvectors}
\label{ssc:CBeig}

First, we give the definition of combinatorial eigenvalues
for tensor pairs. Recall the Hadamard product $\circ$
as in \S\ref{sc:prelim}.

\begin{definition}  \label{df:comb:eig}
Let $\mA, \mB \in \mathrm{T}^m(\cpx^n)$ be tensors.
If there exist a number $\lmd \in \cpx$ and a vector $x \in \cpx^n\setminus \{0\}$ such that
\be \label{cbeig:lmd:x}
x \circ ( \mA x^{m-1} - \lmd \mB x^{m-1} ) = 0,
\ee
then $\lmd$ (resp., $x$) is called a combinatorial eigenvalue
(resp., combinatorial eigenvector) of the pair $(\mA, \mB)$.
Such $(\lmd, x)$ is called a combinatorial eigenpair.
\end{definition}

For convenience of writing,  the
combinatorial eigenvalues (resp., eigenvectors, eigenpairs)
defined in \reff{cbeig:lmd:x} are called
CB-eigenvalues (resp., CB-eigenvectors, CB-eigenpairs).
In particular, C-eigenvalues (resp., C-eigenvectors, C-eigenpairs)
as in \reff{TEiCP}
are also CB-eigenvalues (resp., CB-eigenvectors, CB-eigenpairs).

For a subset $J = \{i_1, \ldots, i_k \}  \subseteq [n]$,
denote $x_J = (x_{i_1}, \ldots, x_{i_k})$.
For a tensor $\mF \in \mathrm{T}^m(\cpx^n)$,
let $\mF_J$ be the principal sub-tensor of $\mF$ corresponding to
the set $J$, i.e., $\mF_J$ is a tensor in $\mathrm{T}^m(\cpx^k)$
indexed by $(j_1, \ldots, j_m)$ such that
\[
(\mF_J)_{j_1,\ldots, j_m} = \mF_{j_1, \ldots , j_m}, \quad
j_1, \ldots, j_m \in J.
\]
Similar to $\mF x^{m-1}$, $\mF_{J} (x_J)^{m-1}$ is defined to be the
$k$-dimensional vector, indexed by $j \in J$ such that
\be \label{def:mFJ(x)}
\big( \mF_{J} (x_J)^{m-1} \big)_j  = \sum_{ i_2, \ldots, i_m \in J}
\mF_{j, i_2, \ldots, i_m} x_{i_2} \cdots x_{i_m}.
\ee
Like \reff{df:R(F)}, let $R_J(\mF)$ be the resultant of
the homogeneous tuple $\mF_{J} (x_J)^{m-1}$
\be \label{defi:RJ(mFx)}
R_J(\mF) \, : = \,  Res \big( \mF_{J} (x_J)^{m-1} \big).
\ee

When $(\mA, \mB)$ is a nonsingular tensor pair,
Ling et al.~\cite[Theorem~4.1]{Ling2015}
gave an upper bound for the number of C-eigenvalues.
We give a similar result for CB-eigenvalues.
Furthermore, we also give upper bound for the number of
CB-eigenvectors (up to scaling), for generic tensors $\mA, \mB$.
Thus, the number of C-eigenvectors (up to scaling)
can also be bounded.

\begin{theorem}\label{th:number}
Let $\mathcal{A},\mathcal{B} \in \mathrm{T}^m(\cpx^n) $.

\bit

\item [(i)] If $R_{J}(\mB) \ne 0$ for each $\emptyset \neq J \subseteq [n]$,
then $(\mA, \mB)$ has at most $n m^{n-1}$ CB-eigenvalues.

\item [(ii)] If $\mA,\mB$ are generic tensors in $\mathrm{T}^m(\cpx^n)$, then,
for each CB-eigenvalue, there is a unique CB-eigenvector (up to scaling).

\eit
\end{theorem}
\begin{proof}
(i) This can be done by following the approach
in the proof of Theorem~4.1 of \cite{Ling2015}.
Suppose $\lmd$ is a CB-eigenvalue,
with the CB-eigenvector $u \ne 0$ such that
\[
u \circ (\mA u^{m-1} -\lmd \mB u^{m-1} ) = 0.
\]
Let $J = \{j: u_j \ne 0 \}$, a nonempty set.
Then, the above implies that
\[
  \mathcal{A}_J (u_J)^{m-1} - \lambda \mathcal{B}_J (u_J)^{m-1} = 0.
\]
So, $\lmd$ is an eigenvalue of the sub-tensor pair $(\mA_J, \mB_J)$.
By Theorem~\ref{thm:nb:eigvec}(i), $(\mA_J, \mB_J)$ has at most
$|J| (m-1)^{|J|-1}$ eigenvalues.
By enumerating all possibilities of $J$,
the number of CB-eigenvalues of
$(\mathcal{A},\mathcal{B})$ is at most the number
{\small
\[
\sum_{|J|=1}^{n} {n \choose |J|} |J| (m-1)^{|J|-1}=n m^{n-1}.
\]
}

(ii) When $\mA,\mB$ are generic in the space $\mathrm{T}^m(\cpx^n)$,
for each $\emptyset \ne I \subseteq [n]$,
the subpair $(\mA_I, \mB_I)$ is also generic in $\mathrm{T}^m(\cpx^{|I|})$.
Hence, $(\mA_I, \mB_I)$ has a unique eigenvector (up to scaling) for each eigenvalue,
by Theorem~\ref{thm:nb:eigvec}(ii).
For each CB-eigenpair $(\lmd, u)$ of $(\mA, \mB)$,
we showed in the item (i) that $\lmd$ is an eigenvalue of the sub-tensor pair
$(\mA_J, \mB_J)$ with the eigenvector $u_J$,
with the index set $J=\{j: u_j \ne 0 \}$.

Suppose $v$ is another CB-eigenvector associated to $\lmd$.
Let $I=\{j: v_j \ne 0 \}$. Clearly, $\lmd$ is also an eigenvalue of
the sub-tensor pair $(\mA_I, \mB_I)$.
We show that $I=J$.
Define the set
\[
V = \{ \mc{C} \in \mathrm{T}^m(\cpx^n): \,
R_I(\mc{C}_I) =  R_J(\mc{C}_J) = 0 \}.
\]
The polynomial $R_I(\mc{C}_I)$ is irreducible
in the entries of the subtensor $\mc{C}_I$. The same is true for $R_J(\mc{C}_J)$.
When $I \ne J$, the dimension of the set $V$
is at most $\dim \big( \mathrm{T}^m(\cpx^n) \big)-2$.
When $\mA, \mB$ are generic tensors, the line in the space $\mathrm{T}^m(\cpx^n)$
\[
\mathscr{L} = \{ \mA - \lmd \mB: \, \lmd \in \cpx \}
\]
does not intersect $V$. Therefore, if $I \ne J$,
then $\lmd$ cannot be a common eigenvalue of the two different sub-tensor pairs
$(\mA_I, \mB_I)$ and $(\mA_J, \mB_J)$.
Hence, $I=J$ and $u_J, v_J$ are both eigenvectors of $(\mA_J, \mB_J)$.
By Theorem~\ref{thm:nb:eigvec}(ii), $u$ is a scaling of $v$.
\end{proof}

\section{TEiCPs with strict copositivity}
\label{sect:cop}

In this section, we discuss how to compute
C-eigenvalues of a tensor pair $(\mA, \mB)$ when
$\mB$ is strictly copositive. Note that
$\mB \in \mathrm{T}^m(\re^n)$ is said to be copositive (resp., strictly copositive)
if $\mathcal{B} x^m \geq 0$ (resp., $\mathcal{B} x^m > 0$)
for all $x \in \mathbb{R}^{n}_+ \setminus \{0\}$.
Recall that $(\lmd, x)$ is a C-eigenpair of $(\mA, \mB)$ if
$x$ is a nonzero vector and
\[
0 \leq x \perp ( \lambda \mB x^{m-1} - \mA x^{m-1} )\geq 0.
\]
Any positive scaling of such $x$ is also a C-eigenvector.
When $\mB$ is strictly copositive, we can always scale $x$
such that $\mB x^m = 1$.
Under this normalization, the C-eigenpair $(\lmd, x)$ satisfies
\[
0 =  x^T(  \lambda \mB x^{m-1} - \mA x^{m-1} ) =
\lambda \mB x^m  - \mA x^m = \lmd - \mA x^m.
\]
So, we get $\lmd = \mA x^m$. The C-eigenvalues of $(\mA, \mB)$
can be found by solving the polynomial system
\be \label{tgeicpcop0}
\left\{ \baray{l}
\mathcal{B}x^m=1, \, x \circ ((\mA x^m ) \mB x^{m-1} -\mA x^{m-1} )=0, \\
x\geq 0, \,  (\mA x^m ) \mB x^{m-1} -\mA x^{m-1} \geq 0,
\earay\right.
\ee
where $\circ$ denotes the Hadamard product of two vectors. If we define
\[
a(x) := x \circ \mA x^{m-1},
\quad b(x) :=  x \circ \mB x^{m-1}.
\]
Then, it clearly holds that
\[
x \circ ( \lambda \mathcal{B} x^{m-1} - \mathcal{A} x^{m-1} )=
\lmd b(x) - a(x).
\]
The polynomial system \eqref{tgeicpcop0} can be rewritten as
\be \label{tgeicpcop}
\left\{ \baray{l}
\mathcal{B}x^m=1, \, (\mA x^m ) b(x) - a(x) = 0, \\
x\geq 0, \,   (\mA x^m ) \mB x^{m-1} - \mA x^{m-1} \geq 0.
\earay\right.
\ee
When $\mB$ is strictly copositive,
the solution set of \eqref{tgeicpcop} is compact, because
$\{x\in \mathbb{R}^n: \mathcal{B}x^m=1, x\geq 0\}$ is compact.
The tensor pair $(\mA, \mB)$ has at least one C-eigenvalue
when $\mB$ (or $-\mB$) is strictly copositive
(cf.~\cite[Theorem~2.1]{Ling2015}). Moreover, under some generic
conditions on $\mB$, $(\mA, \mB)$ has finitely many C-eigenvalues
(cf.~Theorem~\ref{th:number}).
They can be ordered monotonically as
\be \label{seq:lmd}
\lambda_1 < \lambda_2 < \cdots < \lambda_N.
\ee

For convenience, denote the polynomial tuples
\be \label{df:f0:sq}
\left\{ \baray{l}
f_0 = \mA x^m, \\
 p = \Big( \mB x^m - 1,  (\mA x^m) b(x) - a(x)  \Big), \\
 q = \Big( x, \,  (\mA x^m) \mB x^{m-1} - \mA x^{m-1} \Big).
\earay \right.
\ee

\subsection{The first C-eigenvalue}

The first eigenvalue $\lambda_1$
equals the optimal value of the optimization problem
\be  \label{larg:cop}
   \left\{ \begin{array}{rl}
  \lambda_1=\min  &  f_0(x)\\
  \mbox{s.t.} & p(x) = 0, \, q(x) \geq 0.
 \end{array}\right.
\ee
We apply Lasserre type semidefinite relaxations \cite{Lasserre01}
to solve \reff{larg:cop}. For the orders $k=m, m+1, \ldots$,
the $k$-th Lasserre relaxation is
\be \label{SDPpoly1:cop}
   \left\{ \begin{array}{lcl}
  \nu_{1,k} := & \min   & \langle f_0,y \rangle\\
  &\mbox{s.t.} &  \langle 1, y \rangle=1, \, L^{(k)}_{p} (y) = 0, \\
  &            &  M_k(y)\succeq 0, \, L^{(k)}_{q} (y) \succeq 0, \,
   y \in \mathbb{R}^{\mathbb{N}^{n}_{2k}}.
 \end{array}\right.
\ee
In the above, $\langle 1, y\rangle=1$ means that
the first entry of $y$ is one, and the matrices
$M_k(y)$, $L^{(k)}_{p} (y)$, $L^{(k)}_{q} (y)$ are defined
as in \reff{moment:mat}-\reff{mat:locliz}.
Its dual problem is
 \be \label{SDPdual1:cop}
   \left\{ \begin{array}{lcl}
  \tilde \nu_{1,k} := &\max & \gamma \\
  &\mbox{s.t.} & f_{0} - \gamma \in I_{2k}(p)+Q_{k}(q).
 \end{array}\right.
\ee
Suppose $y^{1,k}$ is an optimizer of \eqref{SDPpoly1:cop}.
If, for some $t \in [m,k]$, the truncation
$\hat y=y^{1,k}|_{2t}$ (see \reff{trun:z|t}) satisfies
\be \label{rankconal}
\rank \, M_{t-m}(\hat y) \, = \, \rank\, M_{t} (\hat y),
\ee
then $\nu_{1,k}=\lambda_1$
and we can get $\rank\, M_{t} (\hat y)$ global optimizers
of \eqref{larg:cop} (cf. \cite{Nie2}).

\subsection{The second and other eigenvalues}

We discuss how to compute $\lmd_i$ for $i=2,\ldots, N$.
Suppose $\lambda_{i-1}$ is already computed.
We need to determine the next C-eigenvalue $\lambda_i$.
Consider the optimization problem
\begin{equation}\label{N1larg2:cop}
   \left\{ \begin{array}{cl}
 \min   &  f_0(x)\\
  \mbox{s.t.} & p(x)=0, \,
    q(x)\geq 0,\, f_0(x) - \lambda_{i-1} - \delta \geq 0.
 \end{array}\right.
\end{equation}
The optimal value of \eqref{N1larg2:cop} is equal to $\lambda_i$ if
\begin{equation}\label{del2:cop}
  0< \delta < \lambda_i -\lambda_{i-1}.
\end{equation}
Similarly, Lasserre type semidefinite relaxations can be applied to solve
\eqref{N1larg2:cop}. For the orders $k=m, m+1, \ldots$,
the $k$-th Lasserre relaxation is
\be \label{SDPpolyN2:cop}
   \left\{ \begin{array}{lcl}
 \nu_{i,k} :=& \min & \langle f_0 ,z \rangle\\
  &\mbox{s.t.} &  \langle 1,z \rangle=1,  L^{(k)}_{p} (z) = 0, M_k(z)\succeq 0, \\
  &            & L^{(k)}_{q} (z) \succeq 0,
  L^{(k)}_{f_0-\lmd_{i-1} - \dt} (z) \succeq 0,
   z \in \mathbb{R}^{\mathbb{N}^{n}_{2k}}.
 \end{array}\right.
\ee
The dual problem of \eqref{SDPpolyN2:cop} is
\be \label{SDPdualN2:cop}
   \left\{ \begin{array}{lcl}
  \tilde \nu_{i,k}  :=  &\max  & \gamma  \\
  &\mbox{s.t.} & f_{0} - \gamma \in I_{2k}(p)+Q_{k}(q, f_0-\lambda_{i-1}-\delta).
 \end{array}\right.
\ee
Suppose $y^{i,k}$ is an optimizer of \eqref{SDPpolyN2:cop}.
If a truncation $\hat y=y^{i,k}|_{2t}$
satisfies \reff{rankconal} for some $t \in [m,k]$,
then $\nu_{i,k}=\lambda_i$ and we can get optimizers of
\eqref{N1larg2:cop} (cf. \cite{Nie2}).

In practice, the existence of $\lambda_i$ is usually not known in advance.
Even if it exists, its value is typically not available.
So, we need to determine the value of
$\delta$ satisfying \eqref{del2:cop}.
Consider the polynomial optimization problem:
\be \label{k1clarg:cop}
   \left\{ \begin{array}{lcl}
  \tau: =&\max  & f_0(x)\\
  &\mbox{s.t.} & p(x)=0, \,  q(x)\geq 0,\, f_{0}(x) \leq \lambda_{i-1} + \delta.
 \end{array}\right.
\ee
Its optimal value $\tau$ can be computed by Lasserre
relaxations like \eqref{SDPpolyN2:cop}-\eqref{SDPdualN2:cop}.
As in Proposition~\ref{pro:cop:dlta}, $\delta$ satisfies \eqref{del2:cop}
if and only if $\tau = \lmd_{i-1}$. When $\tau = \lmd_{i-1}$,
$\lmd_i$ does not exist if and only if \eqref{SDPpolyN2:cop}
is infeasible for some $k$.

\subsection{An algorithm for computing C-eigenvalues}

Assume that the tensor $\mB$ is strictly copositive.
So, the C-eigenvectors can be normalized as $\mB x^m = 1$.
We propose an algorithm to compute the C-eigenvalues sequentially,
from the smallest one $\lmd_1$ to the biggest one $\lmd_N$.
Since $\mB$ is strictly copositive, $\lmd_1$ always exists \cite{Ling2015}.
We assume there are finitely many C-eigenvalues.

First, we compute $\lambda_1$ by solving semidefinite relaxations
\eqref{SDPpoly1:cop}-\eqref{SDPdual1:cop}.
After getting $\lambda_1$, we solve
\eqref{SDPpolyN2:cop}-\eqref{SDPdualN2:cop} for $\lmd_2$.
If $\lambda_2$ does not exist, then $\lambda_1$ is the
biggest C-eigenvalue and we stop;
otherwise, we continue to determine $\lambda_3$.
Repeating this procedure, we can get all the C-eigenvalues of $(\mA,\mB)$.

\balg \label{Algorithmcop}
For two tensors $\mA,\mB \in \mathrm{T}^m(\re^n)$
with $\mB$ strictly copositive,
compute a set $\Lambda$ of all C-eigenvalues
and a set $U$ of C-eigenvectors, for the pair $(\mA,\mB)$.
Let $U := \emptyset$, $\Lambda := \emptyset$, $i :=1$,  $k:=m$.

\bit

\item [Step~1.] Solve \reff{SDPpoly1:cop} with the order $k$
for an optimizer $y^{1,k}$.

\item [Step~2.] If \reff{rankconal} is satisfied
for some $t \in [m,k]$, then update $U := U \, \cup \, S$,
with $S$ a set of optimizers of \reff{larg:cop};
let $\lmd_1 = \nu_{1,k}$, $\Lambda := \{ \lmd_1 \}$,
$i:=i+1$ and go to Step~3.
If such $t$ does not exist, let $k:=k+1$ and go to Step~1.

\item [Step~3.] Let $\dt = 0.05$, and compute the optimal value
$\tau$ of \eqref{k1clarg:cop}.
If $\tau >  \lmd_{i-1}$, let $\delta:=  \delta/2$
and compute $\tau$ again.
Repeat this, until we get $\tau = \lmd_{i-1}$.
Let $k:=m$.

\item [Step~4.]  Solve \reff{SDPpolyN2:cop} with the order $k$.
If it is infeasible, then \eqref{tgeicpcop}
has no further C-eigenvalues, and stop.
Otherwise, compute an optimizer $y^{i,k}$ for \reff{SDPpolyN2:cop}.

\item [Step~5.] If \reff{rankconal} is satisfied for some
$t\in [m,k]$, then update $U := U \, \cup \, S$
where $S$ is a set of optimizers of \reff{N1larg2:cop};
let $\lmd_i = \nu_{i,k}$, $\Lambda := \Lambda \cup \{ \lmd_i \}$,
$i:=i+1$ and go to Step~3.
If such $t$ does not exist, let $k:=k+1$ and go to Step~4.

\eit

\ealg

The semidefinite relaxation \eqref{SDPpoly1:cop} can be
solved by the software {\tt GloptiPoly~3} \cite{HenrionJJ}
and {\tt SeDuMi} \cite{Sturm}.
When \eqref{rankconal} holds, it can be shown that
$\lambda_{i,k}=\lambda_i$, and we can get a set of optimizers
of \reff{larg:cop}, \reff{N1larg2:cop}.
Such optimizers are the associated eigenvectors for
the C-eigenvalue $\lmd_i$. In Steps~2 and 5,
the method in Henrion and Lasserre \cite{HenrionJ}
can be used to compute the set $S$.

\subsection{Properties of relaxations}

First, we discuss when Algorithm~\ref{Algorithmcop}
has finite convergence. For the polynomial tuple $p$, denote the sets
\be \label{V:CR:(p)}
V_{\cpx}(p) := \{ u \in \cpx^n \, \mid \, p(u) = 0 \}, \quad
V_{\re}(p) := V_{\cpx}(p) \cap \re^n.
\ee

\begin{theorem}\label{conlarg1:cop}
Let $\mathcal{A},\mathcal{B}\in \mathrm{T}^m(\mathbb{R}^{n})$.
Suppose $\mathcal{B}$ is strictly copositive.
Then, we have the properties:

\begin{enumerate}

\item[(i)] The smallest C-eigenvalue $\lambda_1$ of $(\mA, \mB)$
always exists. Moreover, if the set $V_{\re}(p)$ is finite,
then for all $k$ sufficiently large,
\[
 \nu_{1,k}= \tilde \nu_{1,k}= \lambda_1
\]
and the condition \reff{rankconal} must be satisfied.

\item[(ii)] For $i\geq 2$, suppose $\lmd_i$ exists and
$0 < \dt < \lmd_i - \lmd_{i-1}$.
If the set $V_{\re}(p)$ is finite,
then for all $k$ sufficiently large,
   \[
  \nu_{i,k}= \tilde \nu_{i,k}= \lambda_i
   \]
and the condition \reff{rankconal} must be satisfied.

\end{enumerate}
\end{theorem}

\begin{proof}
(i) Since $\mB$ is strictly copositive, $(\mA, \mB)$
has at least one C-eigenvalue (cf.~\cite[Theorem~2.1]{Ling2015}).
So, $\lmd_1$ always exists.
If $V_{\re}(p)$ is finite, the equation $p(x)=0$
has finitely many real solutions. Thus, when the relaxation order $k$
is sufficiently large, we must have
$\nu_{1,k}= \tilde \nu_{1,k}= \lambda_1$
and the flat truncation condition \reff{rankconal}
must be satisfied. This can be implied by
Proposition~4.6 of \cite{LLR08} and Theorem~1.1 of \cite{Nie-rvarPOP}.

(ii) If $0 < \dt < \lmd_i - \lmd_{i-1}$ holds,
the optimal value of \reff{N1larg2:cop} is equal to $\lmd_i$.
When $V_{\re}(p)$ is finite, the equation $p(x)=0$
has finitely many real solutions. The conclusion can be implied by
Proposition~4.6 of \cite{LLR08} and Theorem~1.1 of \cite{Nie-rvarPOP}.
\end{proof}

\begin{remark}
In Theorem~\ref{conlarg1:cop}, if $V_{\re}(p)$ is not a finite set,
$\nu_{1,k}$ and $\tilde{\nu}_{1,k}$ may not have finite convergence to $\lmd_1$,
but the asymptotic convergence can be established.
When $\mB$ is strictly copositive, the set
$\{x\in \mathbb{R}^n : \mB x^m =1, x \geq 0 \}$
is compact, say, contained in the ball
$\{ x\in \mathbb{R}^n : M-x^Tx \geq 0\}$, where $M>0$ is a sufficiently large number.
If we add $M-x^Tx$ to the polynomial tuple $q$,
then $\nu_{1,k}$ and $\tilde{\nu}_{1,k}$ have asymptotic convergence to $\lmd_1$.
This is because such $Q(q)$ is archimedean,
and the asymptotic convergence can be implied by the results in \cite{Lasserre01}.
\end{remark}

However, interestingly, the set $V_{\re}(p)$ is finite
for generic tensors $\mA,\mB$.

\begin{prop}
Let $p$ be as in \reff{df:f0:sq}.
If $\mA,\mB$ are generic tensors, then
$V_{\cpx}(p)$ and $V_{\re}(p)$ are finite sets.
\end{prop}
\begin{proof}
The equation $p(x)=0$ implies that
\[
\mB x^m = 1, \quad a(x) - (\mA x^m) b(x)
= x \circ ( \mA x^{m-1} - (\mA x^m) \mB x^{m-1}  ) = 0.
\]
So, $x$ must be a nonzero vector.
Let $J=\{j: x_j \ne 0 \}$. Then we get
\[
\mA_J (x_J)^{m-1} - (\mA x^m) \mB_J (x_J)^{m-1} =0.
\]
Hence, $x_J$ is an eigenvector of the sub-tensor pair
$(\mA_J, \mB_J)$.
When $\mA,\mB$ are generic, such $x$ must be finitely many,
by Theorem~\ref{th:number}(ii).
The conclusion holds over the complex field.
So, $V_{\cpx}(p)$, as well as $V_{\re}(p)$, is finite, for generic $\mA, \mB$.
\end{proof}

The existence of $\lambda_i$ and the relation \eqref{del2:cop} can be checked as follows.

\begin{proposition}  \label{pro:cop:dlta}
Let $\mA,\mB \in \mathrm{T}^m(\mathbb{R}^{n})$.
Suppose $\mathcal{B}$ is strictly copositive.
Let $\Lambda$ be the set of all C-eigenvalues of $(\mA,\mB)$.
Assume $\Lambda$ is finite.
Let $\lambda_i$ be the $i$-th smallest C-eigenvalue of $(\mA, \mB)$,
and $\lmd_{max}$ be the maximum of them.
For all $i\geq 2$ and all $\delta > 0$, we have the following properties:
\bnum

\item [(i)] If \reff{SDPpolyN2:cop} is infeasible for some $k$, then
$\Lambda \cap [\lmd_{i-1}+ \dt, \infty) = \emptyset$.

\item [(ii)] If $\Lambda \cap [\lmd_{i-1}+ \dt, \infty) = \emptyset$
and $V_{\re}(p)$ is finite, then
\reff{SDPpolyN2:cop} must be infeasible for some $k$.

\item[(iii)] If $\tau = \lambda_{i-1}$ and $\lambda_i$ exists,
then $\dt$ satisfies \eqref{del2:cop}.

\item[(iv)] If $\tau = \lambda_{i-1}$ and \eqref{SDPpolyN2:cop}
is infeasible for some $k$, then $\lambda_i$ does not exist.

\enum
\end{proposition}

\begin{proof}

Since $\mB$ is strictly copositive, every C-eigenvector $x$ can be scaled such that
$\mB x^m = 1$.

(i) Note that, for every eigenpair $(\lmd, u)$ of $(\mA, \mB)$ with
$\lmd \geq \lmd_{i-1}+\dt$, the tms $[u]_{2k}$
(see the notation in \S\ref{sc:prelim})
is always feasible for \reff{SDPpolyN2:cop}.
If \reff{SDPpolyN2:cop} is infeasible for some $k$, then
$(\mA,\mB)$ clearly has no C-eigenvalues $\geq \lmd_{i-1}+ \dt$.

(ii) Suppose $(\mA,\mB)$ has no C-eigenvalues $\geq \lmd_{i-1}+ \dt$
and $V_{\re}(p)$ is finite.
The feasible set of \reff{N1larg2:cop} is empty.
By the Positivstellensatz (cf.~\cite[Theorem~4.4.2]{BCR}), we have
\[
-2 =  \phi + \psi, \quad
\phi \in I(p), \, \psi \in Pr(q, f_0-\lmd_{i-1}-\dt),
\]
where $Pr(q, f_0-\lmd_{i-1}-\dt)$ denotes the preodering generated by the tuple
$(q, f_0-\lmd_{i-1}-\dt)$.
(We refer to \cite{BCR} for preorderings.)
Since $V_{\re}(p)$ is finite, the ideal $I(p)$ is archimedean.
(This is because $-\|p\|^2$ belongs to $I(p)$
and the set $\{x \in \re^n: -\| p \|^2 \geq 0 \}$ is compact.)
So, $I(p)+Q(q, f_0-\lmd_{i-1}-\dt)$ is also archimedean.
Note that $1+\psi$ is strictly positive on
$\{x\in \mathbb{R}^n: p=0, q\geq 0, f_0-\lmd_{i-1}-\dt\geq 0\}$.
By Putinar's Positivstellensatz (cf.~\cite{Putinar}),
$ 1 + \psi \in I(p)+Q(q, f_0-\lmd_{i-1}-\dt)$.
Thus, we get
\[
-1 =  \phi + \sig, \quad
\phi \in I_{2k}(p), \, \sig \in Q_k(q, f_0-\lmd_{i-1}-\dt)
\]
where $\sig = 1 + \psi$ and $k$ is sufficiently large.
This implies that \reff{SDPdualN2:cop} has an improving direction
and it is unbounded from the above.
By weak duality, the relaxation
\reff{SDPpolyN2:cop} must be infeasible, for $k$ big enough.

(iii) If $\tau = \lambda_{i-1}$, then the maximum C-eigenvalue,
which is less than or equal to $\lmd_{i-1}+\dt$,
is still $\lmd_{i-1}$. So, if $\lmd_i$ exists,
we must have $\lmd_i > \lmd_{i-1} + \dt$, i.e.,
 \eqref{del2:cop} is satisfied.

(iv) When \eqref{SDPpolyN2:cop} is infeasible for some $k$,
$(\mA, \mB)$ has no C-eigenvalues $\geq \lmd_{i-1}+\dt$.
So, if $\tau = \lambda_{i-1}$, $\lmd_{i-1}$
is the maximum C-eigenvalue, and $\lmd_i$ does not exist.
\end{proof}

\section{Solving general TEiCPs}
\label{sc:gTEiCP}

In this section, we discuss how to compute complementarity eigenvalues
of $(\mA, \mB)$ for generic tensors
$\mathcal{A},\mathcal{B}\in \mathrm{T}^m(\mathbb{R}^{n})$. Recall that
$\lmd$ is a C-eigenvalue of $(\mA,\mB)$ if there exists a nonzero vector
$x \in \re^n$ such that
\[
0\leq x \perp  (\lambda \mB x^{m-1} - \mA x^{m-1} ) \geq 0.
\]

\subsection{Polynomial optimization reformulations}
As in \S\ref{sect:cop}, we still denote
\[
a(x) := x \circ \mA x^{m-1}, \quad b(x) :=  x \circ \mB x^{m-1}.
\]
If we normalize $x$ to have unit length, then
$(\lmd, x)$ is a C-eigenpair of $(\mA,\mB)$ if and only if
it is a solution of the polynomial system
\be \label{tgeicpvec}
\left\{ \baray{l}
  x^Tx=1,\,\,     \lmd b(x) - a(x) =0,\\
   x\geq 0, \ \  \lambda \mB x^{m-1} - \mA x^{m-1} \geq 0.
\earay \right.
\ee
When $b(x) \neq 0$, the equation $a(x) - \lambda b(x)=0$
holds if and only if
\[
\rank \, \bbm a(x) &  b(x) \ebm \leq 1,
\]
which is equivalent to that
\begin{align}\label{eq3.6}
a(x)_i b(x)_j - b(x)_i a(x)_j =0 \quad
(1\leq i < j \leq n).
\end{align}
Suppose \reff{tgeicpvec} has finitely many real solutions.
For a generic vector $\xi \in \mathbb{R}^{n}$,
we have $\xi^T b(x)\neq 0$ for all $x$ satisfying \reff{tgeicpvec} and
\be \label{eigval}
\lambda= \frac{\xi^T a(x)}{\xi^T b(x)}.
\ee
The C-eigenvalues of $(\mA,\mB)$ can be computed in two cases.

\textbf{Case I: $\xi^T b(x)>0$.}
In this case, the system \eqref{tgeicpvec} is equivalent to
\be \label{case1}
\left\{ \begin{array}{l}
x^T x=1,    a(x)_i b(x)_j - b(x)_i a(x)_j =0 \, (1\leq i < j \leq n), \\
x \geq 0, \, \xi^T b(x)\geq 0, \,
\big( \xi^T a(x) \mathcal{B}x^{m-1} - \xi^T b(x) \mathcal{A}x^{m-1} \big) \geq 0.
 \end{array}\right.
\ee
Note that \reff{case1} does not use $\lmd$ directly.
For generic $(\mA,\mB)$, \eqref{case1} has finitely many solutions.
Once a solution $x$ is found, the C-eigenvalue $\lmd$ can be computed by \eqref{eigval}.
The system \eqref{case1} can be solved as a polynomial optimization problem.
Generate a random polynomial $f(x)\in \mathbb{R}[x]_{2m}$.
Consider the optimization problem
\be \label{larg:vec2}
\left\{ \begin{array}{cl}
  \min   &  f(x)\\
  \mbox{s.t.} & h(x)=0, \,\, g(x)\geq 0,
\end{array}\right.
\ee
where the polynomial tuples $h,g$ are given as
\be \label{def:gh:case1}
\left\{ \baray{l}
h(x) =\Big( x^T x-1, \, \big(
a(x)_i b(x)_j - b(x)_i a(x)_j \big)_{1\leq i < j \leq n} \Big),  \\
g(x)= \Big(x, \, \xi^T b(x), \,
\xi^T a(x) \mathcal{B}x^{m-1} - \xi^T b(x) \mathcal{A}x^{m-1} \Big).
\earay \right.
\ee
Clearly, $x$ satisfies \reff{case1} if and only if
$x$ is feasible for \reff{larg:vec2}.

\textbf{Case II: $\xi^T b(x)<0$.}
In this case, the system \reff{tgeicpvec} is equivalent to
\be \label{case2}
\left\{ \begin{array}{l}
x^T x=1,    a(x)_i b(x)_j - b(x)_i a(x)_j =0 \, (1\leq i < j \leq n), \\
x \geq 0, \, -\xi^T b(x)\geq 0, \,
\xi^T b(x) \mathcal{A}x^{m-1} \, - \, \xi^T a(x) \mathcal{B}x^{m-1} \geq 0.
 \end{array}\right.
\ee
Like \reff{case1}, the system \reff{case2} does not use $\lmd$ directly.
Once a point $x$ satisfying \reff{case2} is obtained,
the C-eigenvalue $\lmd$ can be obtained by \reff{eigval}.
Clearly, $x$ satisfies \reff{case2}
if and only if it is feasible for the optimization problem
\be \label{larg:vec2-2}
   \left\{ \begin{array}{cl}
  \min  &  f(x)\\
  \mbox{s.t.} & h(x)=0,\, \, \tilde g(x)\geq 0,
 \end{array}\right.
\ee
where $h$ is the same as in \reff{def:gh:case1} while
the tuple $\tilde g$ is given as
\be \label{tldg:case2}
\tilde g(x)= \Big(x, \, -\xi^T b(x), \,
\xi^T b(x) \mathcal{A}x^{m-1}- \xi^T a(x) \mathcal{B}x^{m-1} \Big).
\ee
The feasible sets of \eqref{larg:vec2} and \reff{larg:vec2-2}
are compact, since they are contained in the unit sphere.
However, they are possibly empty.

The C-eigenpairs $(\lmd,x)$
satisfying \reff{tgeicpvec} can be found
by computing feasible points of the optimization problems
\reff{larg:vec2} and \reff{larg:vec2-2}.
When the number of C-eigenvectors (normalized to have unit lengths)
is finite, we can compute all the feasible points of
\reff{larg:vec2} and \reff{larg:vec2-2}.
In the following subsections, we show how to do this.

\subsection{Compute C-eigenvectors}

Assume that there are finitely many
C-eigenvectors (normalized to have unit lengths) for the tensor pair $(\mA,\mB)$.
We propose an algorithm for computing all of them.

\subsubsection{C-eigenpairs for case I} \quad

We discuss how to compute the C-eigenvectors satisfying \reff{case1}.
Assume the feasible set of \reff{larg:vec2} is nonempty and finite.
When it is generically chosen, $f$ achieves different
values at different feasible points of \reff{larg:vec2}, say,
they are monotonically ordered as
\be \label{seq:f(1):i=1:N1}
f^{(1)}_1  < f^{(1)}_2   < \cdots < f^{(1)}_{N_1}.
\ee
We aim to compute the C-eigenvectors,
along with the values $f^{(1)}_i$,
in the order $i=1,\ldots, N_1$.
Choose a number $\ell_i$ such that
\be \label{cho:ell-i:caseI}
 f^{(1)}_{i-1} < \ell_i < f^{(1)}_i.
\ee
(For the case $i=1$, $f^{(1)}_0$ can be chosen to be
any value smaller than $f^{(1)}_1$.)
Note that $f^{(1)}_i$ is equal to the optimal value of
\be \label{min:fhg:i-th}
   \left\{ \begin{array}{cl}
  \min   &  f(x)\\
  \mbox{s.t.} & h(x)=0,\, \,  g(x)\geq 0, \,\, f(x) - \ell_i \geq 0.
 \end{array}\right.
\ee
We apply Lasserre type semidefinite relaxations to solve \reff{min:fhg:i-th}.
For the orders $k = m, m+1, \ldots$, the $k$-th Lasserre relaxation is
\be \label{min<f,z>:mom:I}
   \left\{ \begin{array}{lcl}
 \mu_{1,k}  :=  &\min    & \langle f,y \rangle\\
  & \mbox{s.t.} &  \langle 1, y\rangle=1,  \, L^{(k)}_{h} (y) = 0, \,  M_k(y)\succeq 0, \\
  &            &    L^{(k)}_{g} (y) \succeq 0,
  L^{(k)}_{f-\ell_i} (y) \succeq 0,  \, y \in \mathbb{R}^{\mathbb{N}^{n}_{2k}}.
 \end{array}\right.
 \ee
(See \S\ref{ssc:momloc} for the notation in the above.)
The dual problem of \reff{min<f,z>:mom:I} is
 \begin{equation}  \label{max:gm:sos:I}
   \left\{ \begin{array}{lcl}
  \tilde{\mu}_{1,k}  := &  \max  & \gamma \\
  &\mbox{s.t.} & f - \gamma \in I_{2k}(h)+Q_{k}(g, f-\ell_i),
 \end{array}\right.
\end{equation}
where $I_{2k}(h)$ and $Q_{k}(g, f-\ell_i)$ are defined as in
\eqref{Ik}-\eqref{Qk}.
By weak duality, it can be shown that (cf.~\cite{Lasserre01})
\begin{align}\label{weakd}
\tilde{\mu}_{1,k} \, \leq \, \mu_{1,k} \, \leq \, f^{(1)}_i, \quad  \forall \, k\geq m.
\end{align}
Moreover, both $\{\mu_{1,k}\}$ and $\{\tilde{\mu}_{1,k}\}$
are monotonically increasing.

When \eqref{case1} has a solution, the semidefinite relaxation \reff{min<f,z>:mom:I}
is always feasible. Suppose $y^{i,k}$ is an optimizer of \reff{min<f,z>:mom:I}.
If for some $t \in [m, k]$, the truncation $\hat y :=y^{i,k}|_{2t}$
satisfies the rank condition
\be \label{flat:trun:caseI}
\rank \, M_{t-m}(\hat y) \, = \,\rank \, M_{t} (\hat y),
\ee
then one can show that $\mu_{1,k}=\tilde{\mu}_{1,k}=f^{(1)}_i$
and we can get $\rank\, M_{t} (\hat y)$ optimizers
of \eqref{min:fhg:i-th} (cf.~\cite{Nie2}).
The method in \cite{HenrionJ} can be applied to compute the
minimizers of \reff{min:fhg:i-th}.
Interestingly, we will show that the rank condition
\reff{flat:trun:caseI} must be satisfied,
for generic tensors $\mA, \mB$
(cf.~Theorem~\ref{thm:infeas:ficvg}).

\subsubsection{C-eigenpairs for case II} \quad

Now we show how to find the C-eigenvectors satisfying \reff{case2}.
The computation is similar to the case I.
Assume the feasible set of \reff{larg:vec2-2} is nonempty and finite.
Order its objective values monotonically as
\be \label{sq:f(2):1:N2}
f^{(2)}_1  < f^{(2)}_2   < \cdots < f^{(2)}_{N_2}.
\ee
We compute the C-eigenvectors and the value $f^{(2)}_i$
in the order $i=1,\ldots, N_2$.
Choose a number $\tilde \ell_i$ such that
\be \label{val:ell-i:c2}
 f^{(2)}_{i-1} < \tilde \ell_i < f^{(2)}_i.
\ee
(For $i=1$, choose $f^{(2)}_0 $ to be any value
smaller than $f^{(2)}_1$.)
Note that $f^{(2)}_i$ is equal to the minimum value of
\be \label{minf:tdg:no:i}
  \left\{ \begin{array}{cl}
  \min  &  f(x)\\
  \mbox{s.t.} & h(x)=0,\, \,  \tilde g(x)\geq 0, \,\, f(x) - \tilde \ell_i \geq 0.
 \end{array}\right.
\ee
For an order $k \geq m$, the $k$-th Lasserre relaxation
(cf.~\cite{Lasserre01}) for solving \reff{minf:tdg:no:i} is
\begin{equation}\label{mfz:mom:c2}
   \left\{ \begin{array}{lcl}
 \mu_{2,k} :=  &\min   & \langle f,z \rangle\\
  & \mbox{s.t.} &  \langle 1,z \rangle=1,  L^{(k)}_{h} (z) = 0, \, M_k(z)\succeq 0, \\
  &            &    L^{(k)}_{\tilde g} (z) \succeq 0, \,
  L^{(k)}_{f- \tilde \ell_i} (z) \succeq 0, \,
  z \in \mathbb{R}^{\mathbb{N}^{n}_{2k}}.
 \end{array}\right.
 \end{equation}
Its dual optimization problem is
 \begin{equation}  \label{mgm:sos:c2}
   \left\{ \begin{array}{lcl}
  \tilde{\mu}_{2,k}  :=  &  \max  & \gamma \\
  &\mbox{s.t.} & f - \gamma \in I_{2k}(h)+Q_{k}(\tilde g, f- \tilde \ell_i).
 \end{array}\right.
\end{equation}

Suppose $z^{i,k}$ is an optimizer of \reff{mfz:mom:c2}.
If for some $t \in [m, k]$, the truncation $\hat z :=z^{i,k}|_{2t}$
satisfies the rank condition
\be \label{ftcd:cas2}
\rank \, M_{t-m}(\hat z) \, = \,\rank \, M_{t} (\hat z),
\ee
then $\mu_{2,k}=\tilde{\mu}_{2,k}=f^{(2)}_i$
and we can get $\rank \, M_{t} (\hat z)$ optimizers of
\eqref{minf:tdg:no:i} (cf.~\cite{Nie2}).
We will show that the condition \reff{ftcd:cas2} must be satisfied
for generic tensors $\mA, \mB$
(cf.~Theorem~\ref{thm:infeas:ficvg}).

\subsubsection{An algorithm for computing C-eigenpairs} \quad

In practice, the $f, \ell_i, \tilde \ell_i$ need to be chosen properly.
We propose to choose $f$ in the form as
\be \label{f=sos:RR}
f = [x]_m^T (R^TR) [x]_m,
\ee
where $R$ is a random square matrix.
For $f$ as in \reff{f=sos:RR}, we almost always have
\[
f^{(1)}_1 > 0, \quad f^{(2)}_1 >0.
\]
Thus, we can choose
\be \label{eq:-1:0}
f^{(1)}_0 = f^{(2)}_0  = -1, \quad  \ell_1 = \tilde \ell_1 = 0.
\ee

In the computation of $f^{(1)}_i, f^{(2)}_i$, suppose
the values of $f^{(1)}_{i-1}, f^{(2)}_{i-1}$ are already computed.
In practice, for $\dt > 0$ small enough, we can choose
\[
\ell_i =  f^{(1)}_{i-1} + \dt, \qquad
\tilde \ell_i =  f^{(2)}_{i-1} + \dt,
\]
to satisfy \reff{cho:ell-i:caseI} and \reff{val:ell-i:c2}.
Such value of $\dt$ can be determined
by solving the following maximization problems:
\be \label{k1clarg:vec}
\left\{ \begin{array}{lcl}
\theta_1 = & \max   & f(x)\\
  &\mbox{s.t.} & h(x)=0,\, g(x)\geq 0, \, f(x) \leq f^{(1)}_{i-1} +  \delta,
 \end{array}\right.
\ee
\be  \label{theta2:cas2}
\left\{ \begin{array}{lcl}
\theta_2 =  & \max   & f(x)\\
  &\mbox{s.t.} & h(x)=0,\, \tilde g(x)\geq 0, \, f(x) \leq f^{(2)}_{i-1} +  \delta.
 \end{array}\right.
\ee
Their optimal values can be computed by Lasserre type semidefinite relaxations.
When $h(x)=0$ has finitely many real solutions,
we must have $\theta_1 = f^{(1)}_{i-1}$ and $\theta_2 = f^{(2)}_{i-1}$,
for $\dt>0$ sufficiently small.
For such case, the relations \reff{cho:ell-i:caseI}
and \reff{val:ell-i:c2} will be satisfied.
This is justified by Lemma~\ref{lm:cho:dt:gen}.

Note that $f$ achieves only finitely many values
in the feasible sets of \reff{larg:vec2}, \reff{larg:vec2-2},
when $(\mA, \mB)$ has finitely many normalized C-eigenvectors.

\balg
\label{Algorithmgen}
For two given tensors $\mA, \mB \in \mathrm{T}^m(\mathbb{R}^{n})$,
compute a set $\Lambda$ of C-eigenvalues,
and a set $U$ of C-eigenvectors, for the pair $(\mA,\mB)$.

\bit

\item [Step~0.] Choose $f$ as in \reff{f=sos:RR},
with $R$ a random square matrix.
Choose a random vector $\xi \in \re^n$.
Let $U = \emptyset$, $i=1$, $k=m$, $\ell_1 = 0$, $\tilde \ell_1 =0$.

\item [Step~1.] Solve \reff{min<f,z>:mom:I} for the order $k$.
If it is infeasible, then \eqref{case1} has no
further C-eigenvectors (except those in $U$);
let $k=m$, $i=1$ and go to Step~4.
Otherwise, compute an optimizer $y^{i,k}$ for \reff{min<f,z>:mom:I}.

\item [Step~2.] If \reff{flat:trun:caseI} is satisfied
for some $t \in [m,k]$, then update $U := U \, \cup \, S$,
where $S$ is a set of optimizers of \reff{min:fhg:i-th};
let $i:=i+1$ and go to Step~3.
If such $t$ does not exist, let $k:=k+1$ and go to Step~1.

\item [Step~3.] Let $\dt = 0.05$, and compute the optimal value
$\theta_1$ of \eqref{k1clarg:vec}.
If $\theta_1 >  f^{(1)}_{i-1}$, let $\delta:=  \delta/2$
and compute $\theta_1$ again.
Repeat this process, until $\theta_1 = f^{(1)}_{i-1}$ is met.
Let $\ell_i := f^{(1)}_{i-1} + \dt$, $k=m$,
then go to Step~1.

\item [Step~4.] Solve \reff{mfz:mom:c2} for the order $k$.
If it is infeasible, then \eqref{case2} has no
further C-eigenvectors (except those in $U$) and go to Step~7.
Otherwise, compute an optimizer $z^{i,k}$ for it.

\item [Step~5.] Check whether or not \reff{ftcd:cas2}
is satisfied for some $t \in [m,k]$.
If yes, update $U := U \, \cup \,S$,
where $S$ is a set of optimizers of \reff{minf:tdg:no:i};
let $i:=i+1$ and go to Step~6.
If no, let $k:=k+1$ and go to Step~4.

\item [Step~6.] Let $\dt = 0.05$, and compute the optimal value
$\theta_2$ of \reff{theta2:cas2}.
If $\theta_2 >  f^{(2)}_{i-1}$, let $\delta:=  \delta/2$
and compute $\theta_2$ again.
Repeat this process, until we get $\theta_2 = f^{(2)}_{i-1}$.
Let $\tilde \ell_i = f^{(2)}_{i-1} + \dt$, $k=m$,
and go to Step~4.

\item [Step~7.] Let
$\Lambda := \{ \xi^T a(u) / \xi^T b(u): \, u \in U \}.$

\eit
\ealg

The Lasserre type semidefinite relaxations
\reff{min<f,z>:mom:I} and \reff{mfz:mom:c2} can be
solved by the software
GloptiPoly 3 \cite{HenrionJJ} and SeDuMi \cite{Sturm}.
In Step~2 and Step~5, the method in Henrion and Lasserre \cite{HenrionJ}
can be used to compute optimizers of \reff{min:fhg:i-th}.
The same is true for \reff{minf:tdg:no:i}
and its Lasserre relaxation \reff{mfz:mom:c2}.

\subsection{Properties of the relaxations}

First, we prove that Algorithm~\ref{Algorithmgen}
converges in finitely many steps for generic tensors $\mA,\mB$.
Let $T_1, T_2$ be the feasible sets of
\reff{min:fhg:i-th} and \reff{minf:tdg:no:i}, respectively.
Let $V_{\re}(h)$ be defined as in \reff{V:CR:(p)}.

\begin{theorem}  \label{thm:infeas:ficvg}
Let $\mathcal{A},\mathcal{B}\in \mathrm{T}^m(\mathbb{R}^{n})$
be two tensors. Let $h, g, \tilde{g}$ be the polynomial tuples
as in \reff{def:gh:case1}, \reff{tldg:case2},
constructed from $\mA, \mB$ and a vector $\xi \in \re^n$.
Then, for all $\ell_i, \tilde \ell_i$
satisfying \reff{cho:ell-i:caseI} and \reff{val:ell-i:c2},
we have the following properties:
 \begin{enumerate}

\item[(i)]  The relaxation \reff{min<f,z>:mom:I}
is infeasible for some order $k$ if and only if
the feasible set $T_1$ of \reff{min:fhg:i-th} is empty.

\item[(ii)] Suppose $T_1 \ne \emptyset$.
If $V_{\re}(h)$ is a finite set, then
for $k$ sufficiently large, the rank condition
\reff{flat:trun:caseI} must be satisfied and
\[
\mu_{1,k} = \tilde{\mu}_{1,k} =  f^{(1)}_i.
\]

\item[(iii)]  The relaxation \reff{mfz:mom:c2}
is infeasible for some order $k$ if and only if
the feasible set $T_2$ of \reff{minf:tdg:no:i} is empty.

\item[(iv)] Suppose $T_2 \ne \emptyset$.
If $V_{\re}(h)$ is a finite set, then
for $k$ sufficiently large, the rank condition
\reff{ftcd:cas2} must be satisfied and
\[
\mu_{2,k} = \tilde{\mu}_{2,k} =  f^{(2)}_i.
\]

\end{enumerate}
\end{theorem}
\begin{proof}
(i) ``only if" direction: If the relaxation \reff{min<f,z>:mom:I}
is infeasible for some order $k$, then the feasible set of
\reff{min:fhg:i-th} must be empty. This is because, if otherwise
\reff{min:fhg:i-th} has a feasible point, say, $u$,
then the tms $[u]_{2k}$ (see the notation in \S\ref{sc:prelim})
generated by $u$
must be feasible for \reff{min<f,z>:mom:I}.

``if" direction: Since $T_1 = \emptyset$,
by the Positivstellensatz (cf.~\cite[Theorem~4.4.2]{BCR}), we have
\[
-2 = \phi + \psi, \quad \phi \in I(h), \, \psi \in Pr(g, f-\ell_i).
\]
Here, $Pr(g, f-\ell_i)$ is the preordering of the tuple $(g, f-\ell_i)$.
(We refer to \cite{BCR} for preorderings.)
Note that the sum $1+\psi$ is strictly positive
on $\{x\in \mathbb{R}^n: h = 0, g\geq 0,  f-\ell_i \geq 0\}$.
The ideal $I(h)$ is archimedean, because $1-\|x\|^2 \in I(h)$.
So, $I(h) + Q(g, f-\ell_i)$ is also archimedean.
By Putinar's Positivstellensatz, $1 + \psi \in I(h) + Q(g, f-\ell_i)$.
This implies that
\[
-1 = \phi + \sig, \quad \phi \in I_{2k}(h), \, \sig \in Q_k(g, f-\ell_i),
\]
where $\sig = 1 + \psi$ and $k$ is sufficiently large.
So, the dual optimization problem \reff{max:gm:sos:I}
has an improving direction and it is unbounded from the above.
By weak duality, the optimization \reff{min<f,z>:mom:I}
must be infeasible.

(ii) When the set $V_{\re}(h)$ is finite,  the Lasserre's hierarchy
\reff{min<f,z>:mom:I}-\reff{max:gm:sos:I}
must have finite convergence, and the condition
\reff{flat:trun:caseI} must be satisfied,
when $k$ is sufficiently large.
This can be implied by Theorem~1.1 of \cite{Nie-rvarPOP}
and Proposition~4.6 of \cite{LLR08}.

(iii)-(iv):  These two items can be proved exactly in the same way
as for (i)-(ii). The proof is omitted here, for cleanness of the paper.
\end{proof}

\begin{remark}
In Theorem~\ref{thm:infeas:ficvg}(ii), (iv), if $V_{\re}(h)$ is not finite,
then we can only get the asymptotic convergence
\[
\lim_{ k \to \infty} \mu_{1,k} = \lim_{ k \to \infty} \tilde{\mu}_{1,k} =  f^{(1)}_i,
\qquad
\lim_{ k \to \infty} \mu_{2,k} = \lim_{ k \to \infty} \tilde{\mu}_{2,k} =  f^{(2)}_i.
\]
This is because $V_{\re}(h)$ is contained in the unit sphere
$\{x\in \mathbb{R}^n : x^Tx=1\}$ and the ideal $I(h)$ is archimedean.
The asymptotic convergence can be implied by \cite{Lasserre01}.
However, the set $V_{\re}(h)$ is finite for generic tensors
$\mA,\mB$, as shown below.
\end{remark}

\begin{prop} \label{V(h):finite}
Let $h$ be as in \reff{def:gh:case1}.
If $\mA, \mB$ are generic tensors,
then $V_{\cpx}(h)$ and $V_{\re}(h)$ are finite sets.
\end{prop}
\begin{proof}
By the construction of $h$ as in \reff{def:gh:case1},
$h(x)=0$ if and only if
\be \label{|x|=1:rk[ab]=1}
x^Tx - 1 =0, \quad \rank \, \bbm a(x) & b(x) \ebm \leq 1.
\ee
Let
$J = \{ j : \, x_j \ne 0 \}$.
We claim that $b(x) \ne 0$. Suppose otherwise $b(x)=0$, then
\[
\mB_J  (x_J)^{m-1} = 0.
\]
(See \S\ref{ssc:CBeig} for the notation $\mB_J$.)
Since $x_J$ is a nonzero vector, we get $R_J(\mB) = 0$.
This is impossible, when $\mB$ is a generic tensor.
Thus, in \reff{|x|=1:rk[ab]=1}, $b(x) \ne 0$
and there exists $\lmd$ such that
\[
a(x) - \lmd b(x) = 0.
\]
Thus, we get that
\[
x \circ ( \mA x^{m-1} - \lmd \mB x^{m-1} ) = a(x) - \lmd b(x) = 0.
\]
This implies that $x$ is a C-eigenvector, associated to $\lmd$.
By Theorem~\ref{th:number}, there are finitely many
normalized C-eigenvetors, when $\mA, \mB$ are generic.
Therefore, $h(x)=0$ has finitely many complex solutions,
for generic $\mA,\mB$. So, both $V_{\cpx}(h)$ and $V_{\re}(h)$ are finite.
\end{proof}

\begin{prop} \label{pro:T1T2:epty}
Let $T_1$ (resp., $T_2$) be the feasible set of
\reff{min:fhg:i-th} (resp., \reff{minf:tdg:no:i}).
For all $\xi \in \re^n$, we have the properties:
\bit

\item [(i)] If $T_1 = \emptyset$, then
there is no C-eigenvector $x$ satisfying
\reff{case1} and $f(x) \geq \ell_i$.

\item [(ii)] If $T_2 = \emptyset$, then
there is no C-eigenvector $x$ satisfying
\reff{case2} and $f(x) \geq \tilde \ell_i$.

\item [(iii)] For the case $i=1$, if $T_1 = \emptyset$ then
the set \reff{case1} is empty;
if $T_2 = \emptyset$ then
the set \reff{case2} is empty.
Thus, if $T_1 = T_2 = \emptyset$,
then the pair $(\mA, \mB)$ has no C-eigenpairs.

\eit
\end{prop}
\begin{proof}
For every C-eigenpair $(\lmd, x)$,
it holds that $a(x) - \lmd b(x) = 0$, so
\[
\xi^T a(x) - \lmd \xi^T b(x) = 0.
\]
If $\xi^T b(x) >0$, $x$ satisfies \reff{case1}.
If $\xi^T b(x) <0$, $x$ satisfies \reff{case2}.
If $\xi^T b(x) =0$, then $\xi^T a(x) = 0$ and
$x$ satisfies both \reff{case1} and \reff{case2}.

(i) Every C-eigenvector $x$ satisfying
\reff{case1} and $f(x) \geq \ell_i$
belongs to the set $T_1$. So, if $T_1 = \emptyset$,
then no C-eigenvector $x$ satisfies
\reff{case1} and $f(x) \geq \ell_i$.

(ii) The proof is same as for the item (i).

(iii) For the case $i=1$, the set $T_1$ is same as \reff{case1},
and $T_2$ is same as \reff{case2}, because
$\ell_1 \leq f^{(1)}_1$ and $\tilde \ell_1 \leq f^{(2)}_1$.
So, if $T_1 = \emptyset$, then \reff{case1} is empty;
if $T_2 = \emptyset$, then \reff{case2} is empty.
If $T_1 = T_2 = \emptyset$, then \reff{case1} and \reff{case2}
are both empty, i.e., $(\mA, \mB)$ has no C-eigenpairs.
\end{proof}

\begin{lemma}  \label{lm:cho:dt:gen}
Assume that $V_{\re}(h)$ is a finite set.
Let $\theta_1,\theta_2$ be as in
\reff{k1clarg:vec}, \reff{theta2:cas2}.
Then, for $\dt>0$, $\ell_i =f^{(1)}_{i-1} + \dt$
satisfies \reff{cho:ell-i:caseI} if and only if $\theta_1 = f^{(1)}_{i-1}$,
and $\tilde \ell_i =f^{(2)}_{i-1} + \dt$  satisfies \reff{val:ell-i:c2}
if and only if $\theta_2 = f^{(2)}_{i-1}$.
\end{lemma}
\begin{proof}
Since $V_{\re}(h)$ is a finite set,
\reff{larg:vec2} has finitely many objective values on its feasible set,
and they can be ordered as in \reff{seq:f(1):i=1:N1}.
The optimal value $\theta_1$ of \reff{k1clarg:vec} is
the maximum objective value of \reff{larg:vec2}
that is less than or equal to $f^{(1)}_{i-1}+ \dt$.
Then, \reff{cho:ell-i:caseI} is satisfied if and only if $\theta_1 = f^{(1)}_{i-1}$.
The proof is same for the case of $\tilde \ell_i$.
\end{proof}

\section{Numerical Experiments}
\label{sc:numexp}

In this section, we present numerical experiments for solving
tensor eigenvalue complementarity problems.
The Lasserre type semidefinite relaxations are solved by
the software
GloptiPoly 3 \cite{HenrionJJ} and SeDuMi \cite{Sturm}.
The experiments are implemented on
a laptop with an Intel Core i5-2520M CPU (2.50GHz)
and 8GB of RAM, using Matlab R2014b.
We display 4 decimal digits for numerical numbers.

We use $\mc{I}$ to denote the identity tensor
(i.e., $\mc{I}_{i_1\cdots i_m}= 1$ if $i_1=\cdots = i_m$,
and $\mc{I}_{i_1\cdots i_m}= 0$ otherwise).
When $\mB$ is strictly copositive.
Algorithm~\ref{Algorithmcop} is applied to solve the TEiCP;
otherwise, Algorithm~\ref{Algorithmgen} is used.

\bex  \label{Exmp:LHQ}
(i) (\cite[Example 5.1]{Ling2015}).
Consider the tensors $\mA, \mB \in \mathrm{T}^4(\mathbb{R}^{2})$ given as
{\tiny
 \[
 \begin{array}{cc}
   \mA(:,:,1,1)=\left(
                \begin{array}{lr}
                  0.8147 &0.5164\\
                  0.5164& 0.9134 \\
                \end{array}
              \right), &

\mA(:,:,1,2)=\left(
  \begin{array}{cc}
            0.4218 &0.8540\\
            0.8540& 0.9595 \\
  \end{array}
\right),\\
\mA(:,:,2,1)=
\left(
  \begin{array}{cc}
    0.4218& 0.8540\\
    0.8540 &0.9595 \\
  \end{array}
\right), &
\mA(:,:,2,2)=
\left(
  \begin{array}{cc}
    0.6787 &0.7504\\
    0.7504 &0.3922 \\
  \end{array}
\right), \\
\mB(:,:,1,1)=
\left(
  \begin{array}{cc}
    1.6324 &1.1880\\
    1.1880 &1.5469 \\
  \end{array}
\right), &
\mB(:,:,1,2)=
\left(
  \begin{array}{cc}
    1.6557& 1.4424\\
    1.4424& 1.9340 \\
  \end{array}
\right), \\
\mB(:,:,2,1)=
\left(
  \begin{array}{cc}
    1.6557 &1.4424\\
    1.4424 &1.9340 \\
  \end{array}
\right), &
\mB(:,:,2,2)=
\left(
  \begin{array}{cc}
    1.6555& 1.4386\\
    1.4386 &1.0318 \\
  \end{array}
\right).
 \end{array}
\] \noindent}The
tensor $\mB$ is strictly copositive,
beause all its entries are positive.
By Algorithm~\ref{Algorithmcop}, we get three
C-eigenpairs $(\lmd_i, u_i)$:
\[\begin{array}{lll}
    \lambda_1=0.4678 , & u_1=(0.8328,    0.0585), \\
    \lambda_2=0.4848 , & u_2=(0.2577,    0.6538),\\
    \lambda_3=0.4991 , & u_3=(0.8847,    0.0000).
  \end{array}
\]
The computation takes about $2$ seconds.

\noindent
(ii) (\cite[Example 5.2]{Ling2015}).
Consider the tensors $\mathcal{A},\mathcal{B}\in \mathrm{T}^4(\mathbb{R}^{3})$ given as:
 {\tiny
 \[
   \mA(:,:,1,1)=\left(
                \begin{array}{ccc}
                  0.6229 &0.2644& 0.3567\\
                  0.2644 &0.0475 &0.7367\\
                  0.3567 &0.7367 &0.1259 \\
                \end{array}
              \right), \,
\mA(:,:,1,2)=\left(
  \begin{array}{ccc}
    0.7563 &0.5878& 0.5406\\
0.5878& 0.1379 &0.0715\\
0.5406 &0.0715& 0.3725 \\
  \end{array}
\right),
\]
\[
\mA(:,:,1,3)=
\left(
  \begin{array}{ccc}
    0.0657 &0.4918 &0.9312\\
0.4918 &0.7788 &0.9045\\
0.9312 &0.9045 &0.8711 \\
  \end{array}
\right), \,
\mA(:,:,2,1)=
\left(
  \begin{array}{ccc}
    0.7563 &0.5878& 0.5406\\
0.5878& 0.1379& 0.0715\\
0.5406 &0.0715& 0.3725\\
  \end{array}
\right),
\]
\[
\mA(:,:,2,2)=
\left(
  \begin{array}{ccc}
    0.7689 &0.3941 &0.6034\\
0.3941& 0.3577& 0.3465\\
0.6034& 0.3465& 0.4516 \\
  \end{array}
\right), \,
\mA(:,:,2,3)=
\left(
  \begin{array}{ccc}
    0.8077 &0.4910 &0.2953\\
0.4910& 0.5054 &0.5556\\
0.2953& 0.5556& 0.9608 \\
  \end{array}
\right),
\]
\[
\mA(:,:,3,1)=
\left(
  \begin{array}{ccc}
    0.0657& 0.4918& 0.9312\\
0.4918 &0.7788& 0.9045\\
0.9312 &0.9045& 0.8711 \\
  \end{array}
\right), \,
\mA(:,:,3,2)=
\left(
  \begin{array}{ccc}
    0.8077& 0.4910& 0.2953\\
0.4910 &0.5054& 0.5556\\
0.2953& 0.5556 &0.9608 \\
  \end{array}
\right),
\]
\[
\mA(:,:,3,3)=
\left(
  \begin{array}{ccc}
    0.7581& 0.7205& 0.9044\\
0.7205& 0.0782 &0.7240\\
0.9044 &0.7240 &0.3492 \\
  \end{array}
\right), \,
\mB(:,:,1,1)=
\left(
  \begin{array}{ccc}
    0.6954& 0.4018 &0.1406\\
0.4018 &0.9957& 0.0483\\
0.1406 &0.0483& 0.0988 \\
  \end{array}
\right),
\]
\[
\mB(:,:,1,2)=
\left(
  \begin{array}{ccc}
    0.6730& 0.5351& 0.4473\\
0.5351& 0.2853 &0.3071\\
0.4473& 0.3071 &0.9665 \\
  \end{array}
\right),  \,
\mB(:,:,1,3)=
\left(
  \begin{array}{ccc}
    0.7585& 0.6433& 0.2306\\
0.6433& 0.8986& 0.3427\\
0.2306& 0.3427& 0.5390 \\
  \end{array}
\right),
\]
\[
\mB(:,:,2,1)=
\left(
  \begin{array}{ccc}
    0.6730& 0.5351& 0.4473\\
0.5351& 0.2853 &0.3071\\
0.4473& 0.3071& 0.9665 \\
  \end{array}
\right),  \,
\mB(:,:,2,2)=
\left(
  \begin{array}{ccc}
    0.3608 &0.3914& 0.5230\\
0.3914 &0.6822& 0.5516\\
0.5230 &0.5516& 0.7091\\
  \end{array}
\right),
\]
\[
\mB(:,:,2,3)=
\left(
  \begin{array}{ccc}
    0.4632& 0.2043& 0.2823\\
0.2043 &0.7282& 0.7400\\
0.2823 &0.7400& 0.9369 \\
  \end{array}
\right),
\mB(:,:,3,1)=
\left(
  \begin{array}{ccc}
    0.7585& 0.6433& 0.2306\\
0.6433 &0.8986& 0.3427\\
0.2306 &0.3427& 0.5390 \\
  \end{array}
\right),
\]
\[
\mB(:,:,3,2)=
\left(
  \begin{array}{ccc}
    0.4632 &0.2043& 0.2823\\
0.2043& 0.7282 &0.7400\\
0.2823 &0.7400& 0.9369 \\
  \end{array}
\right), \,
\mB(:,:,3,3)=
\left(
  \begin{array}{ccc}
    0.8200& 0.5914& 0.4983\\
0.5914 &0.0762 &0.2854\\
0.4983& 0.2854 &0.1266 \\
  \end{array}
\right).
 \] \noindent}The
tensor $\mB$ is also strictly copositive,
beause all its entries are positive.
By Algorithm~\ref{Algorithmcop}, we get
three C-eigenpairs $(\lmd_i, u_i)$:
%\textcolor[rgb]{1.00,0.00,0.00}{(times 28.4710s)}
\[\begin{array}{ll}
    \lambda_1=1.5520, & u_1 =(0.2201,    0.1572,    0.8680), \\
    \lambda_2=2.3562, & u_2 =(0.0000,    0.0312,    1.5404),  \\
    \lambda_3=2.7583, & u_3 =(0.0000,    0.0000,    1.6765).
  \end{array}
\]
The computation takes about $15$ seconds.
\eex

\bex \label{Exmp:CYY} (\cite[\S5]{Chen2015}).
Consider the tensors $\mA, \mB \in \mathrm{T}^6(\mathbb{R}^{4})$
with $\mB =\mathcal{I}$ (the identity tensor)
and $\mA$ listed as in Table \ref{ex:CYY:table}.
Note that $\mA$ is a symmetric tensor, i.e.,
$\mA_{i_1i_2i_3i_4i_5i_6} = \mA_{j_1j_2j_3j_4j_5j_6}$
whenever $(i_1,i_2,i_3,i_4,i_5,i_6)$ is a permutation of $(j_1,j_2,j_3,j_4,j_5,j_6)$.
So, only its upper triangular entries are listed.
\begin{table}[htbp]
\caption{The symmetric tensor $\mA \in \mathrm{T}^6(\mathbb{R}^{4})$
in Example~\ref{Exmp:CYY}.}
\begin{center} \footnotesize
\renewcommand{\arraystretch}{1.2}
{\tiny
\begin{tabular}{llrllrllrllr} \hline
$\mA_{111111}$  =\ \ 0.5000,   & $\mA_{111112}$  = -0.2369,    & $\mA_{111113}$  =\ \   0.1953, & $\mA_{111114}$  = -0.2691,\\
$\mA_{111122}$  =\ \ 0.0835,   & $\mA_{111123}$  = -0.2016,    & $\mA_{111124}$  = -0.0441,     & $\mA_{111133}$  =\ \   0.0567,\\
$\mA_{111134}$  = -0.2784,     & $\mA_{111144}$  =\ \  0.2321, & $\mA_{111222}$  = -0.1250,     & $\mA_{111223}$  =\ \   0.0333,\\
$\mA_{111224}$  =\ \ 0.0235,   & $\mA_{111233}$  =\ \  0.0093, & $\mA_{111234}$  = -0.0304,     & $\mA_{111244}$  = -0.0167,\\
$\mA_{111333}$  =\ \ 0.1028,   & $\mA_{111334}$  = -0.0385,    & $\mA_{111344}$  =\ \   0.0068, & $\mA_{111444}$  =\ \   0.1627,\\
$\mA_{112222}$  = -0.1002,     & $\mA_{112223}$  =\ \  0.0733, & $\mA_{112224}$  =\ \   0.0607, & $\mA_{112233}$  = -0.1125,\\
$\mA_{112234}$  =\ \ 0.0096,   & $\mA_{112244}$  =  -0.0810,   & $\mA_{112333}$  = -0.0299,     & $\mA_{112334}$  =\ \   0.0153,\\
$\mA_{112344}$  =\ \ 0.0572,   & $\mA_{112444}$  =\ \  0.0251, & $\mA_{113333}$  =\ \   0.1927, & $\mA_{113334}$  = -0.1024,\\
$\mA_{113344}$  = -0.0885,     & $\mA_{113444}$  =\ \  0.0289, & $\mA_{114444}$  = -0.0668,     & $\mA_{122222}$  = -0.2707,\\
$\mA_{122223}$  = -0.1066,     & $\mA_{122224}$  = -0.1592,    & $\mA_{122233}$  =\ \   0.0805, & $\mA_{122234}$  = -0.0540,\\
$\mA_{122244}$  = -0.0434,     & $\mA_{122333}$  = -0.0048,    & $\mA_{122334}$  = -0.0118,     & $\mA_{122344}$  =\ \   0.0196,\\
$\mA_{122444}$  = -0.0585,     & $\mA_{123333}$  = -0.0442,    & $\mA_{123334}$  = -0.0618,     & $\mA_{123344}$  =\ \   0.0318,\\
$\mA_{123444}$  =\ \ 0.0332,   & $\mA_{124444}$  = -0.2490,    & $\mA_{133333}$  =\ \   0.1291, & $\mA_{133334}$  =\ \   0.0704,\\
$\mA_{133344}$  = -0.0032,     & $\mA_{133444}$  =\ \  0.0270, & $\mA_{134444}$  =\ \   0.0232, & $\mA_{144444}$  = -0.3403,\\
$\mA_{222222}$  = -0.6637,     & $\mA_{222223}$  =\ \  0.2191, & $\mA_{222224}$  =\ \   0.3280, & $\mA_{222233}$  =\ \   0.1834,\\
$\mA_{222234}$  =\ \ 0.0627,   & $\mA_{222244}$  =\ \ 0.0860,  & $\mA_{222333}$  =\ \   0.1590, & $\mA_{222334}$  = -0.0217,\\
$\mA_{222344}$  =\ \ 0.1198,   & $\mA_{222444}$  = -0.1674,    & $\mA_{223333}$  =\ \   0.0549, & $\mA_{223334}$  = -0.0868,\\
$\mA_{223344}$  =\ \ 0.0043,   & $\mA_{223444}$  =\ \  0.0101, & $\mA_{224444}$  = -0.0307,     & $\mA_{233333}$  = -0.3553,\\
$\mA_{233334}$  =\ \ 0.0207,   & $\mA_{233344}$  =\ \  0.1544, & $\mA_{233444}$  = -0.1707,     & $\mA_{234444}$  = -0.3557,\\
$\mA_{244444}$  = -0.1706,     & $\mA_{333333}$  =\ \  0.7354, & $\mA_{333334}$  = -0.3628,     & $\mA_{333344}$  = -0.2650,\\
$\mA_{333444}$  = -0.0479,     & $\mA_{334444}$  = -0.0084,    & $\mA_{344444}$  = -0.0559,     & $\mA_{444444}$  =\ \   0.6136.\\\hline
\end{tabular}
}
\end{center}
\label{ex:CYY:table}
\end{table}
The tensor $\mB$ is strictly copositive.
We apply Algorithm~\ref{Algorithmcop}
and get fifteen C-eigenpairs $(\lmd_i, u_i)$:
\[\begin{array}{rrr}
\lambda_1=&-12.7096, & u_1=(0.7814,    0.7331,    0.7630,    0.8654), \\
\lambda_2=&-9.3276, & u_2=(0.7414,    0.8448,    0.1123,    0.8819),\\
\lambda_3=&-6.9921, & u_3=(0.0000,    0.5798,    0.8395,    0.9214),\\
\lambda_4=&-4.8469, & u_4=(0.7907,    0.0000,    0.8629,    0.8365),\\
\lambda_5=&-3.1530, & u_5=(0.1704,    0.0000,    0.9300,    0.8406),\\
\lambda_6=&-0.9797, & u_6=(0.0000,    0.8032,    0.0000,    0.9492),\\
\lambda_7=&-0.0933, & u_7=(0.4471,    0.0000,    0.0186,    0.9987),\\
\lambda_8=&0.3394, & u_8=(1.0000,    0.0000,    0.0000,    0.1880),\\
\lambda_9=&0.6136, & u_9=(0.0000,    0.0000,    0.0000,    1.0000), \\
\lambda_{10}=&0.9215, & u_{10}=(0.5942,    0.5831,    0.9856,    0.0000),\\
\lambda_{11}=&1.7772, & u_{11}=(0.9431,    0.0000,    0.0000,    0.8165),\\
\lambda_{12}=&3.0313, & u_{12}=(0.0000,    0.9338,    0.8342,    0.0887),\\
\lambda_{13}=&3.1009, & u_{13}=(0.0000,    0.9239,    0.8504,    0.0000),\\
\lambda_{14}=&3.3208, & u_{14}=(0.0000,    0.9619,    0.7672,    0.4016),\\
\lambda_{15}=&4.5057, & u_{15}=(0.8754,    0.0000,    0.9051,    0.0000).
\end{array}
\]
The computation takes about $16083$ seconds.

\eex

\bex \label{Exmp:CQ} (\cite[\S5]{ChenQi2015})
Consider the tensors $\mA, \mB \in \mathrm{T}^6(\mathbb{R}^{4})$
with $\mB =\mathcal{I}$
and $\mA$ listed as in Table \ref{ex:CQ:table}.
The tensor $\mA$ is symmetric, so only the upper triangular entries are listed.
\begin{table}[htbp]
\caption{The symmetric tensor $\mA \in \mathrm{T}^6(\mathbb{R}^{4})$
in Example~\ref{Exmp:CQ}.}
\begin{center} \footnotesize
\renewcommand{\arraystretch}{1.2}
{\tiny
\begin{tabular}{llrllrllrllr} \hline
$\mA_{111111}$  = 0.1197,   & $\mA_{111112}$  = 0.4859,    & $\mA_{111113}$  = 0.4236,     & $\mA_{111114}$  = 0.1775,\\
$\mA_{111122}$  = 0.4639,   & $\mA_{111123}$  = 0.4951,    & $\mA_{111124}$  = 0.5322,     & $\mA_{111133}$  = 0.4219,\\
$\mA_{111134}$  = 0.4606,   & $\mA_{111144}$  = 0.4646,    & $\mA_{111222}$  = 0.4969,     & $\mA_{111223}$  = 0.4649,\\
$\mA_{111224}$  = 0.5312,   & $\mA_{111233}$  = 0.5253,    & $\mA_{111234}$  = 0.4635,     & $\mA_{111244}$  = 0.4978,\\
$\mA_{111333}$  = 0.5562,   & $\mA_{111334}$  = 0.5183,    & $\mA_{111344}$  = 0.4450,     & $\mA_{111444}$  = 0.4754,\\
$\mA_{112222}$  = 0.4992,   & $\mA_{112223}$  = 0.5420,    & $\mA_{112224}$  = 0.4924,     & $\mA_{112233}$  = 0.5090,\\
$\mA_{112234}$  = 0.4844,   & $\mA_{112244}$  = 0.5513,    & $\mA_{112333}$  = 0.5040,     & $\mA_{112334}$  = 0.4611,\\
$\mA_{112344}$  = 0.4937,   & $\mA_{112444}$  = 0.5355,    & $\mA_{113333}$  = 0.4982,     & $\mA_{113334}$  = 0.4985,\\
$\mA_{113344}$  = 0.4756,   & $\mA_{113444}$  = 0.4265,    & $\mA_{114444}$  = 0.5217,     & $\mA_{122222}$  = 0.2944,\\
$\mA_{122223}$  = 0.5123,   & $\mA_{122224}$  = 0.4794,    & $\mA_{122233}$  = 0.5046,     & $\mA_{122234}$  = 0.4557,\\
$\mA_{122244}$  = 0.5332,   & $\mA_{122333}$  = 0.5161,    & $\mA_{122334}$  = 0.5236,     & $\mA_{122344}$  = 0.5435,\\
$\mA_{122444}$  = 0.5576,   & $\mA_{123333}$  = 0.5685,    & $\mA_{123334}$  = 0.5077,     & $\mA_{123344}$  = 0.5138,\\
$\mA_{123444}$  = 0.5402,   & $\mA_{124444}$  = 0.4774,    & $\mA_{133333}$  = 0.6778,     & $\mA_{133334}$  = 0.4831,\\
$\mA_{133344}$  = 0.5030,   & $\mA_{133444}$  = 0.4865,    & $\mA_{134444}$  = 0.4761,     & $\mA_{144444}$  = 0.3676,\\
$\mA_{222222}$  = 0.1375,   & $\mA_{222223}$  = 0.5707,    & $\mA_{222224}$  = 0.5440,     & $\mA_{222233}$  = 0.5135,\\
$\mA_{222234}$  = 0.5770,   & $\mA_{222244}$  = 0.6087,    & $\mA_{222333}$  = 0.5075,     & $\mA_{222334}$  = 0.4935,\\
$\mA_{222344}$  = 0.5687,   & $\mA_{222444}$  = 0.5046,    & $\mA_{223333}$  = 0.5226,     & $\mA_{223334}$  = 0.4652,\\
$\mA_{223344}$  = 0.5289,   & $\mA_{223444}$  = 0.4810,    & $\mA_{224444}$  = 0.5310,     & $\mA_{233333}$  = 0.6187,\\
$\mA_{233334}$  = 0.5811,   & $\mA_{233344}$  = 0.4811,    & $\mA_{233444}$  = 0.4883,     & $\mA_{234444}$  = 0.4911,\\
$\mA_{244444}$  = 0.4452,   & $\mA_{333333}$  = 0.1076,    & $\mA_{333334}$  = 0.6543,     & $\mA_{333344}$  = 0.4257,\\
$\mA_{333444}$  = 0.5786,   & $\mA_{334444}$  = 0.5956,    & $\mA_{344444}$  = 0.4503,     & $\mA_{444444}$  = 0.3840.\\\hline
\end{tabular}
}
\end{center}
\label{ex:CQ:table}
\end{table}
The tensor $\mB$ is copositive. We apply Algorithm~\ref{Algorithmcop}
and get only one C-eigenpair:
\[\begin{array}{rrr}
\lambda_1=&515.4181, & u_1=(0.7909, 0.7957, 0.7941, 0.7941).
\end{array}
\]
The computation takes about $140$ seconds.
\eex

\bex   \label{ExTNW2}
Consider the tensors $\mA, \mB  \in \mathrm{T}^3(\re^n)$ given as:
\[
\mathcal{A}_{ijk}= \frac{(-1)^{j}}{i} + \frac{(-1)^{k}}{j} + \frac{(-1)^{i}}{k},\;
\mathcal{B}=\mathcal{I}.
\]
By Algorithm~\ref{Algorithmcop}, for $n=3$,
we get seven C-eigenpairs $(\lmd_i, u_i)$:
\[\begin{array}{rrr}
\lambda_1=&-8.7329,& u_1 =(0.8432,    0.2568,    0.7266), \\
\lambda_2=&-8.1633, & u_2 =(0.8529,    0.0000,    0.7241),\\
\lambda_3=&-3.1458, & u_3 =(0.9982,    0.1768,    0.0000),\\
\lambda_4=&-3.0000, & u_4 =(1.0000,    0.0000,    0.0000),\\
\lambda_5=&-1.2863, & u_5 =(0.0000,    0.3171,    0.9893),\\
\lambda_6=&-1.0000, & u_6 =(0.0000,    0.0000,    1.0000),\\
\lambda_7=&2.1458, & u_7=(0.3491,    0.9856,    0.0000).
\end{array}
\]
When $n=4$, we get seven C-eigenpairs $(\lmd_i, u_i)$:
%\textcolor[rgb]{1.00,0.00,0.00}{(time 63.1330s)}
\[\begin{array}{rrr}
\lambda_1=&-8.3411,& u_1 =(0.8498,    0.0000,    0.7253,    0.1674), \\
\lambda_2=&-8.1633, & u_2 =(0.8529,    0.0000,    0.7241,    0.0000),\\
\lambda_3=&-3.0413, & u_3 =(0.9996,    0.0000,    0.0000,    0.1043),\\
\lambda_4=&-3.0000, & u_4 =(1.0000,    0.0000,    0.0000,    0.0000),\\
\lambda_5=&-1.0971, & u_5 =(0.0000,    0.0000,    0.9960,    0.2284),\\
\lambda_6=&-1.0000, & u_6 =(0.0000,    0.0000,    1.0000,    0.0000),\\
\lambda_7=&6.6817, & u_7=(0.4382,    0.7963,    0.0000,    0.7434).
\end{array}
\]
For $n=3$, the computation takes about $21$ seconds;
for $n=4$, it takes about $138$ seconds.
When $n=5$, thirteen C-eigenvalues are obtained.
The computer is out of memory for computing the resting C-eigenvalues.
\eex

\bex\label{ExTNZ3}
Consider the tensors $\mA, \mB \in \mathrm{T}^5(\mathbb{R}^{n})$ such that
\[
\mathcal{A}_{i_1 \cdots i_5}=\Big(\sum_{j=1}^5 (-1)^{j+1} \exp(i_j) \Big)^{-1},\;
\mathcal{B}=\mathcal{I}.
\]
By Algorithm~\ref{Algorithmcop}, for $n = 3$,
we get only one C-eigenpair:
\[\begin{array}{lll}
\lambda_1=2.4335, & u_1=(0.7526,    0.6080,    0.9245).
\end{array}
\]
When $n = 4$, we get only one C-eigenpair:
\[\begin{array}{lll}
\lambda_1=5.4419, & u_1=(0.7391,    0.6412,    0.7719,    0.8313).
\end{array}
\]
When $n = 5$, we get only one C-eigenpair:
\[\begin{array}{lll}
\lambda_1=8.8555, & u_1=(0.7347,    0.6513,    0.7212,    0.7404,    0.7585).
\end{array}
\]
For $n=3$, the computation takes about $7$ seconds;
for $n=4$, it takes about $44$ seconds;
for $n=5$, it takes about $2662$ seconds.
\eex

In the following examples, the tensor $\mB$ is not strictly copositive.
So, Algorithm~\ref{Algorithmgen} is applied.

\bex\label{Example6673}
Consider the tensors $\mA, \mB \in \mathrm{T}^3(\re^n)$ given as:
\[
\mathcal{A}_{ijk}=\tan(i - \frac{j}{2} + \frac{k}{3}),\;
\mathcal{B}_{ijk}=\frac{(-1)^{j}}{i} + \frac{(-1)^{k}}{j} + \frac{(-1)^{i}}{k}.
\]
By Algorithm~\ref{Algorithmgen},
for $n = 3$,  we get two C-eigenpairs $(\lmd_i, u_i)$:
%\textcolor[rgb]{1.00,0.00,0.00}{(times 5.6780s)}
\[\begin{array}{rrr}
\lambda_1=&-4.0192, & u_1=(0.5171,    0.8559,    0.0000), \\
\lambda_2=&-0.3669, & u_2=(1.0000,    0.0000,    0.0000).
\end{array}
\]
When $n=4$, we get two C-eigenpairs $(\lmd_i, u_i)$:
%\textcolor[rgb]{1.00,0.00,0.00}{(times 53.1030s)}
\[\begin{array}{lll}
\lambda_1=&-0.8408, & u_1=(0.7095,    0.4519,    0.0000,    0.5407), \\
\lambda_2=&-0.2332, & u_2=(0.9962,    0.0000,    0.0000,    0.0874). \\
\end{array}
\]
When $n=5$, we get six C-eigenpairs $(\lmd_i, u_i)$:
\[\begin{array}{lll}
\lambda_1=&-13.3912, & u_1=(0.0000,    0.0000,    0.0000,    0.3370,    0.9415), \\
\lambda_2=&-4.1204, & u_2=(0.0000,    0.0398,    0.0000,    0.0470,    0.9981), \\
\lambda_3=&-0.8408, & u_3=(0.7095,    0.4519,    0.0000,    0.5407,    0.0000), \\
\lambda_4=&-0.8216, & u_4=(0.7004,    0.4548,    0.0000,    0.5501,    0.0068), \\
\lambda_5=&-0.4376,  & u_5=(0.6150,   0.1435,    0.4245,    0.3803,    0.5257), \\
\lambda_6=&-0.2332,  & u_6=(0.9962,   0.0000,    0.0000,    0.0874,    0.0000). \\
\end{array}
\]
For $n=3$, the computation takes about $2$ seconds;
for $n=4$, it takes about $9$ seconds;
for $n=5$, it takes about $3003$ seconds.
\eex

\bex\label{Example69124}
Consider the tensors $\mA, \mB \in \mathrm{T}^4(\mathbb{R}^{n})$ such that
\[
\baray{l}
\mathcal{A}_{i_1 i_2 i_3 i_4}=\frac{1}{10}(i_1 + 2 i_2 +3i_3 +4i_4
-\sqrt{i_1^2+i_2^2+i_3^2+i_4^2}), \\
\mathcal{B}_{i_1 i_2 i_3 i_4}=\arctan(i_1 i_2 i_3 i_4).
\earay
\]
We apply Algorithm~\ref{Algorithmgen} to compute the C-eigenpairs.
When $n = 3$, we get three C-eigenpairs $(\lmd_i, u_i)$:
\[\begin{array}{lll}
\lambda_1=&0.8706, & u_1=(1.0000,    0.0000,    0.0000), \\
\lambda_2=&0.9780, & u_2=(0.6209,    0.0000,    0.7839), \\
\lambda_3=&1.3163, & u_3=(0.0000,    0.0000,    1.0000).
\end{array}
\]
When $n = 4$, we also get three C-eigenpairs $(\lmd_i, u_i)$:
\[\begin{array}{lll}
\lambda_1=&0.8706, & u_1=(1.0000,    0.0000,    0.0000,    0.0000), \\
\lambda_2=&1.0698, & u_2=(0.7850,    0.0000,    0.0000,    0.6195), \\
\lambda_3=&1.7455, & u_3=(0.0000,    0.0000,    0.0000,    1.0000).
\end{array}
\]
When $n = 5$, we also get three C-eigenpairs $(\lmd_i, u_i)$:
\[\begin{array}{lll}
\lambda_1=&0.8706, & u_1=(1.0000,    0.0000,    0.0000,    0.0000,    0.0000), \\
\lambda_2=&1.1536, & u_2=(0.8527,    0.0000,    0.0000,    0.0000,    0.5224), \\
\lambda_3=&2.1787, & u_3=(0.0000,    0.0000,    0.0000,    0.0000,    1.0000).
\end{array}
\]
For $n=3$, the computation takes about $6$ seconds;
for $n=4$, it takes about $35$ seconds;
for $n=5$, it takes about $716$ seconds.
\eex

\bex\label{Example610144}
Consider the tensors $\mA,\mB \in \mathrm{T}^4(\mathbb{R}^{n})$ such that
\[
 \mathcal{A}_{i_1 i_2 i_3 i_4}=(1+i_1 +2 i_2 +3 i_3 + 4 i_4)^{-1},\;
 \mathcal{B}_{i_1 i_2 i_3 i_4}=\tan(i_1) + \cdots + \tan(i_4).
\]
Algorithm~\ref{Algorithmgen} is applied.
When $n = 3,4,5$, the relaxations \reff{min<f,z>:mom:I}
and \reff{mfz:mom:c2} for the order $k=4$ are both infeasible,
so there are no C-eigenvalues.
For $n=3$, the computation takes about $2$ seconds;
for $n=4$, it takes about $8$ seconds;
for $n=5$, it takes about $43$ seconds.
\eex

\bex\label{Example5Trand35}
Consider two randomly generated tensors
$\mathcal{A},\mathcal{B}\in \mathrm{T}^3(\mathbb{R}^{5})$:
{\tiny
 \[
   A(:,:,1) =
   \left(
     \begin{array}{rrrrr}
    0.0195  & -0.8385  & -0.5971  & -0.9968 &   0.8617\\
    1.2397  &  1.8190  &  0.3261 &  -1.0365&   -0.6295\\
   -0.1187 &   0.2297 &   1.5407 &  -1.0985 &   0.1256\\
    1.3101 &  -0.9982 &   1.1868 &   0.2386 &   2.4171\\
    1.4264  &  2.4354 &  -0.4358  & -1.4201  &  0.5474 \\
     \end{array}
   \right),
 \]
 \[
A(:,:,2) =\left(
     \begin{array}{rrrrr}
    1.0276 &  -1.0345 &  -0.6651 &  -0.7659  &  0.1898\\
    0.3746 &   0.6527 &  -1.1189 &   0.8586  &  0.8419\\
    0.8412 &   0.8611 &   0.0405 &  -0.3752 &  -0.2802\\
    0.0881 &  -1.3853  &  1.8775 &   0.2183 &   0.1735\\
    0.7881 &   0.5900 &   0.0509 &  -0.7750  &  0.0494 \\
     \end{array}
   \right),
\]
 \[
A(:,:,3) =\left(
     \begin{array}{rrrrr}
   -0.5525  & -0.2713 &  -0.3630  &  0.8350&   -0.0891\\
    0.2359  & -1.8207 &   0.6906  & -1.7055&   -1.2772\\
    2.0821  & -1.3487  & -0.4501  &  0.8657&    1.4453\\
   -0.7056  &  0.2588 &  -0.5409  & -0.6727 &  -1.7967\\
    0.7374 &   0.3692 &  -0.8594   & 0.7489 &   0.3929 \\
     \end{array}
   \right),
\]
\[
A(:,:,4) =\left(
     \begin{array}{rrrrr}
     0.4575 &   0.3137 &   0.4335 &   0.9388 &   0.0927\\
   -0.0042  &  1.6425  &  0.8085 &  -1.5722 &  -0.9639\\
   -0.1085  & -0.0514 &  -1.3662  &  0.3091 &  -3.1922\\
   -1.2807  & -0.2399  & -1.1180  & -1.2672 &   0.2671\\
   -1.0279 &  -0.9839 &  -0.3586  &  0.7765 &   0.4211 \\
     \end{array}
   \right),
\]
\[
A(:,:,5) =\left(
     \begin{array}{rrrrr}
     2.0504 &   0.4528 &  -1.7698 &  -2.5073 &  -0.1142\\
   -0.0395  &  0.3460 &  -0.1017  & -1.5303 &   0.1027\\
   -0.4152  & -1.2332 &  -0.1069 &  -1.2440 &   1.6888\\
   -0.8989  & -0.3438&   -2.5825 &  -0.4245&   -0.8625\\
   -1.6842  & -0.7582 &  -1.7254 &  -0.1353 &  -0.0564 \\
     \end{array}
   \right),
\]
\[
B(:,:,1) =\left(
     \begin{array}{rrrrr}
    0.1278  & -1.2405&   -1.5521 &  -0.3097  &  0.4371\\
    1.0476 &  -1.1941 &  -0.1954 &  -1.3133 &   0.3712\\
   -0.8638 &   0.4681 &   0.1090 &   0.8267 &  -0.7007\\
   -1.6955  &  1.0037 &   0.9138 &  -0.0934 &  -0.1997\\
   -0.5110  & -0.2755&   -0.8768  & -0.3897  & -0.2546 \\
     \end{array}
   \right),
\]
\[
B(:,:,2) =\left(
     \begin{array}{rrrrr}
    0.1286  &  0.5255 &   0.3809 &   0.1088 &  -1.2674\\
    0.2852  & -1.1047 &  -0.8320  &  0.9058 &  -2.3433\\
   -1.5964 &   0.3327 &   0.1657  &  0.2164 &   0.4927\\
   -0.9393 &  -0.9674  & -0.4843  &  0.4749  &  0.4720\\
   -0.6881 &   1.7844  &  2.0353  &  0.5464  &  0.7580\\
     \end{array}
   \right),
\]
\[
B(:,:,3) =\left(
     \begin{array}{rrrrr}
    0.4473  &  0.8023  &  2.1941 &   1.7633 &  -2.0100\\
    0.8716  &  0.1619  &  0.0832 &   1.0375 &   1.0234\\
   -0.4001 &   1.0824  &  0.4427 &   1.6162 &   0.1706\\
   -0.2331 &   0.2375  & -0.0875 &  -0.5156  & -1.0727\\
    0.6626 &   0.1542  &  0.3014 &   1.1429  & -0.1337 \\
     \end{array}
   \right),
\]
\[
B(:,:,4) =\left(
     \begin{array}{rrrrr}
    -0.4009 &   0.8938 &   0.5559 &   0.8235  &  0.3279\\
   -0.9912  &  0.3709  &  0.4380 &  -0.2003  & -0.4898\\
   -0.5454  &  2.6579 &   1.1804 &   1.3327 &  -0.2990\\
    0.5980 &   1.1167 &   0.6838  & -0.4269 &   0.2665\\
    0.7989 &  -0.5784&   -1.1768 &  -1.1067 &   0.1850 \\
     \end{array}
   \right),
\]
\[
B(:,:,5) =\left(
     \begin{array}{rrrrr}
   -0.3332 &   0.3925 &   1.5851&   -0.2666  & -0.4003\\
    0.8811  &  0.4142 &  -0.7639 &   0.6644  &  0.0389\\
    0.4362  &  0.3792 &   1.5087 &  -0.6220  & -0.4257\\
    2.3515  & -0.7528 &   0.9182  & -1.3888 &   0.0862\\
    0.8837  & -0.6053 &   2.6629  & -1.9644  & -0.9562 \\
     \end{array}
   \right).
\] \noindent}By
Algorithm~\ref{Algorithmgen}, we get five C-eigenpairs $(\lmd_i, u_i)$:
\[\begin{array}{rrl}
\lambda_1=&-0.3593, & u_1=(0.1195,    0.2810,    0.9522,   0.0000,   0.0000), \\
\lambda_2=&0.0717,  & u_2=(0.8084,   0.0000,    0.3062,    0.4481,    0.2278),\\
\lambda_3=&0.2998,  & u_3=(0.0000,  0.9292,    0.3696,    0.0000,    0.0000),\\
\lambda_4=&0.8616,  & u_4=(0.7547,   0.0000,    0.3079,    0.3919,    0.4267),\\
\lambda_5=&2.1402,  & u_5=(0.7067,    0.3554,    0.3536,    0.2436,    0.4358).
\end{array}
\]
The computation takes about $995$ seconds.
\eex

%%%%%%%%%%%%%%%%%%%%%%%%%%%
%%%%%%%%%%%%%%%%%%%%%%%%%%%

\end{document}